\documentclass[12pt,a4paper]{article}
\usepackage{amsmath, enumerate, fullpage,wasysym, bbm}
\usepackage{amsthm}
\usepackage{amssymb, bm}
\usepackage{amsbsy}
\usepackage{amsfonts}
\usepackage{amstext}
\usepackage{amscd}
\usepackage[dvipdfmx]{graphicx}
\usepackage{hangcaption}

\numberwithin{equation}{section}
\theoremstyle{plain}
\newtheorem{thm}{Theorem}[section]
\newtheorem{prop}[thm]{Proposition}
\newtheorem{cor}[thm]{Corollary}
\newtheorem{lem}[thm]{Lemma}
\theoremstyle{definition}
\newtheorem{exa}[thm]{Example}
\newtheorem{conj}[thm]{Conjecture}
\newtheorem{rem}[thm]{Remark}
\newtheorem{defi}[thm]{Definition}
\newtheorem{prob}[thm]{Problem}
\newcommand{\real}{\mathbb{R}}
\newcommand{\comp}{\mathbb{C}}

\renewcommand{\caption}{\hangcaption}

\def\mrhd{\kern0.2em\rule{0.035em}{0.52em}\kern-.35em\gtrdot \kern-.2em}
\def\submrhd{\mathop{\kern0.0em\lower0.05ex\hbox{\rule{0.03em}{0.39em}}\kern-0.08em\gtrdot \kern0.0em}}
\def\dismrhd{\mathop{{\kern0.1em\lower0.07ex\hbox{\rule{0.035em}{0.57em}}\kern-.1em\gtrdot \kern0.02em}}}

\begin{document}
\title{Free infinite divisibility for beta distributions and related ones}

\author{Takahiro Hasebe \\ Department of Mathematics, Hokkaido University \\ Kita 10, Nishi 8, Kita-ku, Sapporo 060-0810, Japan\footnote{Most of this work was done when the author was in Kyoto University and in University of Franche-Comt\'e in Besan\c{c}on}}   
\date{}
\maketitle

\begin{abstract}
We prove that many of beta, beta prime, gamma, inverse gamma, Student t- and ultraspherical distributions are freely infinitely divisible, but
 some of them are not. The latter negative result follows from a local property of probability density functions. Moreover, we show that
the Gaussian, many of ultraspherical and Student t-distributions have free divisibility indicator 1.  
\end{abstract}



\section{Introduction}
\subsection{Beta and beta prime distributions}
Wigner's semicircle law $\mathbf{w}$ and the Marchenko-Pastur law (or  free Poisson law) $\mathbf{m}$, defined by 
$$
\mathbf{w}(dx) =  \frac{\sqrt{4-x^2}}{2\pi}1_{[-2,2]}(x)\,dx,~~~\mathbf{m}(dx)=\frac{1}{2\pi}\sqrt{\frac{4-x}{x}}1_{[0,4]}(x)\,dx, 
$$
are the most important distributions in free probability because they are respectively the limit distributions of the free central limit theorem and free Poisson's law of small numbers. 
In the context of random matrices, $\mathbf{w}$ and $\mathbf{m}$ are the large $N$ limit of the eigenvalue distributions of $X_N$ and $X_N^2$ respectively, where $X_N$ is an $N\times N$ normalized Wigner matrix.   

Those measures belong to the class of \emph{freely infinitely divisible} (or FID for short) distributions, the main subject of this paper. This class appears as the spectral distributions of large random matrices~\cite{BG05,C05}. 
Research on free probability or more specifically FID distributions has motivated some new directions in classical probability: the upsilon transformation (see \cite{BT06}), type A distributions \cite{ABP09, MPS12} and matrix-valued L\'evy processes~\cite{AM12}. 
Handa \cite{H12} found a connection of branching processes and generalized gamma convolutions (GGCs) to Boolean convolution (see Section \ref{fdi}), a convolution related to free probability.  

Up to affine transformations, $\mathbf{w}$ and $\mathbf{m}$ are special cases of \emph{beta distributions}: 
$$
\bm{\bm{\beta}}_{p,q}(dx) := \frac{1}{B(p,q)} x^{p-1}(1-x)^{q-1}\,1_{[0,1]}(x)\, dx,~~p,q>0, 
$$ 
where $B(p,q)$ is the beta function $\int_0^1x^{p-1}(1-x)^{q-1}\,dx$. 
Moreover, $\bm{\bm{\beta}}_{1/2,1/2}$ is the arcsine law which appears in the monotone central limit theorem \cite{M01} and plays a central role in free type A distributions \cite{ABP09}.  
If $p=q$, the beta distribution $\bm{\bm{\beta}}_{p,p}$ can be shifted to a symmetric measure which is called the \emph{ultraspherical distribution} essentially. 
This family contains Wigner's semicircle law and a symmetric arcsine law. If we take the limit $p \to 0$, a Bernoulli law appears, which is known as the limit distribution of Boolean central limit theorem \cite{SW97}. In the case $p+q=2$, the measure $\bm{\bm{\beta}}_{p,q}$ has explicit Cauchy and Voiculescu transforms \cite{AHb}. Moreover, if we let $\bm{\bm{\beta}}_a :=\bm{\bm{\beta}}_{1-a,1+a}$, $-1< a <1$, it holds that 
\begin{equation}\label{eq001}
(D_b \bm{\bm{\beta}}_{a}) \rhd \bm{\bm{\beta}}_{b} = \bm{\bm{\beta}}_{ab},~a,b\in (-1,1).  
\end{equation}
The binary operation $\rhd$ is \emph{monotone convolution} \cite{M00,F09} and $D_a \mu$ is the \emph{dilation} of a probability measure $\mu$ by $a$: 
$(D_a\mu)(A):=\mu(\frac{1}{a}A)$ for Borel sets $A \subset \real$ and $a \neq 0$.  $D_0\mu$ is defined to be $\delta_0$.

\emph{Beta prime distributions} (or beta distributions of the second kind)  
$$
\bm{\bm{\beta}}'_{p,q}(dx):=\dfrac{1}{B(p,q)} \dfrac{x^{p-1}}{(1+x)^{p+q}}\,1_{[0,\infty)}(x)\, dx,~~p,q>0,
$$
also appear related to free probability.  
The measure $\bm{\beta}_{3/2, 1/2}'$ is a one-sided free stable law with stability index $1/2$; see p.\ 1054 of \cite{BP99}. 
The same measure also appears as the law of an affine transformation of $X^{-1}$ when $X$ follows the free Poisson law $\mathbf{m}$.  
If $X$ follows the semicircle law $\textbf{w}$, then $\frac{1}{X+2}$ follows the beta prime distribution $\bm{\bm{\beta}}'_{3/2,3/2}$ up to an affine transformation. If $X$ follows a Cauchy distribution, i.e.\ a free stable law with stability index $1$, then $X^2$ follows the beta prime distribution $\bm{\bm{\beta}}'_{1/2,1/2}$. 

Thus various beta and beta prime distributions appear in noncommutative probability. One motivation of this paper is to understand free infinite divisibility for these distributions.

\subsection{Gamma, inverse gamma, ultraspherical and t-distributions}
Related to beta and beta prime distributions are \emph{gamma distributions} $\gamma_p$, \emph{inverse gamma distributions} $\gamma_p^{-1}$, \emph{ultraspherical distributions} $\mathbf{u}_p$ and (essentially) \emph{t-distributions} $\mathbf{t}_q$:  
\[
\begin{array}{llllll}
&\bm{\gamma}_p(dx)&:= \dfrac{1}{\Gamma(p )} x^{p-1}e^{-x}\,1_{[0,\infty)}(x)\,dx, &p>0,\\ [13pt]
&\bm{\gamma}^{-1}_p(dx)&:= \dfrac{1}{\Gamma( p)} x^{-p-1}e^{-1/x}\,1_{[0,\infty)}(x)\,dx,&p>0, \\[13pt]
&\mathbf{u}_p(dx)&:= \dfrac{1}{4^{p}B(p+\frac{1}{2}, p+\frac{1}{2})}(1-x^2)^{p-\frac{1}{2}}1_{[-1,1]}(x)\,dx,& p > -\frac{1}{2},\\[13pt]
&\mathbf{t}_q(dx)&:= \dfrac{1}{B(\frac{1}{2},q- \frac{1}{2})} \dfrac{1}{(1+x^2)^q}\,1_{(-\infty,\infty)}(x)\,dx,&q>\frac{1}{2}. 
\end{array}
\]
Note that $\bm{\gamma}^{-1}_{1/2}$ coincides with a classical $1/2$-stable law, called the L\'evy distribution. 

If a random variable $X$ follows a distribution $\mu$, we write $X \sim \mu$. If $X \sim \mu$, the measure $D_a \mu$ coincides with the distribution of $a X$. 
The measures $\bm{\bm{\beta}}_{p,q}, \bm{\bm{\beta}}'_{p,q}, \bm{\gamma}_p, \bm{\gamma}^{-1}_p, \mathbf{t}_q$ satisfy the following relations:

\begin{enumerate}[\rm(1)]

\item If $X \sim \bm{\bm{\beta}}_{p,q}$, then $\frac{X}{1-X} \sim \bm{\bm{\beta}}'_{p,q}$, $\frac{1}{X}-1 \sim \bm{\beta}'_{q,p}$. 

\item $\lim_{q\to \infty}D_{q}\bm{\beta}_{p,q} = \bm{\gamma}_p$ in the sense of weak convergence. 

\item $\lim_{q\to \infty}D_{1/q}\bm{\beta}'_{q, p} = \bm{\gamma}^{-1}_p$ in the sense of weak convergence. 

\item If $X \sim \bm{\gamma}_p$, then $X^{-1}\sim \bm{\gamma}^{-1}_p$. 
\item If $X \sim  \bm{\bm{\beta}}_{p+1/2,p+1/2}$, then $2X-1 \sim \mathbf{u}_p$. 

\item If $X \sim \mathbf{t}_q$, then $X^2 \sim \bm{\beta}'_{1/2, q-1/2}$. 
\end{enumerate}
The infinite divisibility of many of these measures as well as $\bm{\beta}'_{p,q}$ is known in classical probability, but proofs of some of them are not trivial. Bondesson's approach provides us with proofs of the above facts from a more general viewpoint (see \cite{B92}, p.\ 59 and p.\  117), which heavily depends on complex analysis. 
\begin{thm}
The probability measures $\bm{\bm{\beta}}'_{p,q}, \bm{\gamma}_p, \bm{\gamma}^{-1}_p, \mathbf{t}_q$ are infinitely divisible in classical probability for all parameters. 
The probability measures $\bm{\beta}_{p,q}, \mathbf{u}_p$ are not infinitely divisible. 
\end{thm}
The latter statement comes from the fact that infinitely divisible distributions except Dirac measures have unbounded supports.   

A motivation of this paper is to understand the free infinite divisibility for $\bm{\beta}_{p,q}$ and $\bm{\beta}_{p,q}'$ as mentioned. 
Another motivation is the following simple question: 
\begin{itemize}
\item What kind of infinitely divisible distributions in classical probability are FID? 
\end{itemize}
Belinschi et al.\ \cite{BBLS11} showed that the Gaussian is FID, which was quite unexpected because no apparent reason exists to expect this result. 
Other examples are also known in \cite{AHS,AHa, BH}.  In this paper we add more examples from $\bm{\beta}_{p,q}', \bm{\gamma}_p, \bm{\gamma}_p^{-1}, \mathbf{t}_q$. 

The proof of \cite{BBLS11} is based on a first-order differential equation of the Cauchy transform of the Gaussian.
The other motivation of the present paper is to understand the result of \cite{BBLS11} better, i.e.\ to understand the relationship between a first-order differential equation of the Cauchy transform and free infinite divisibility. In fact the Cauchy transforms of distributions $\bm{\beta}_{p,q}, \bm{\beta}_{p,q}', \mathbf{u}_p, \bm{\gamma}_p, \bm{\gamma}^{-1}_p, \mathbf{t}_q$ are all Gauss hypergeometric functions or limits of such and thus satisfy first-order differential equations. We will clarify what property of the Cauchy transform in addition to a first-order differential equation guarantees free infinite divisibility.

\subsection{Main results}
We summarize the known results. It is well known that Wigner's semicircle law and the free Poisson law are FID. The law $\bm{\beta}_a=\bm{\bm{\beta}}_{1-a,1+a}$ is FID if (and only if) $\frac{1}{2} \leq |a| <1$ \cite{AHb}. The free infinite divisibility for ultraspherical distributions $\textbf{u}_p$ was conjectured for $p \geq 1$ in \cite[Remark 4.4]{AP10}, and Arizmendi and Belinschi \cite{AB} showed that the ultraspherical distribution $\mathbf{u}_n$ (and also the beta distribution $\bm{\bm{\beta}}_{\frac{1}{2}, n+\frac{1}{2}}$) is FID for $n=1,2,3,\cdots$. 
For beta prime distributions, $\bm{\bm{\beta}}'_{2/3, 1/2}$ is a free stable law and so is FID \cite[p.\ 1054]{BP99}. The law $\bm{\bm{\beta}}'_{1/2,1/2}$ is also known to be FID because it is the square of a Cauchy distribution \cite{AHS}. 
The t-distribution $\mathbf{t}_q$ is FID for $q=1,2,3,\cdots$ \cite{H}. The chi-square distribution $\frac{1}{\sqrt{\pi x}}e^{-x}1_{[0,\infty)}(x)\,dx$ coincides with $\bm{\gamma}_{1/2}$ and it is FID \cite{AHS}, while the exponential distribution is not FID.\footnote{F.\ Lehner found a negative Hankel determinant of free cumulants of the exponential distribution. See also Section \ref{subsec12}. } 

The main theorem of this paper is the following, which is proved through Sections \ref{sec2}--\ref{st}. 

\begin{thm}\label{thm1}
\begin{enumerate}[\rm(1)]
\item\label{thm1-1} The beta distribution $\bm{\beta}_{p,q}$ is FID in the following cases: (i) $p, q \geq \frac{3}{2}$; (ii) $0<p\leq \frac{1}{2},~p+q \geq 2$; (iii) $0<q\leq \frac{1}{2},~p+q \geq 2$. 
\item The beta distribution $\bm{\beta}_{p,q}$ is not FID in the following cases: (i)  $0<p,q \leq 1$; (ii) $p \in \mathcal{I}$; (iii) $q \in \mathcal{I}$,  where 
$$
\mathcal{I}:=\left(\bigcup_{n=1}^\infty\left(\frac{2n-1}{2n}, \frac{2n}{2n+1}\right) \right) \cup \left( \bigcup_{n=1}^\infty\left(\frac{2n+2}{2n+1}, \frac{2n+1}{2n}\right)\right) \subset \left(\frac{1}{2}, \frac{3}{2}\right).
$$

\item The beta prime distribution $\bm{\beta}'_{p,q}$ is FID if $p\in(0,\frac{1}{2}]\cup [\frac{3}{2},\infty)$. 

\item The beta prime distribution $\bm{\beta}'_{p,q}$ is not FID if $p\in\mathcal{I}$. 

\item The  t-distribution $\mathbf{t}_q$ is FID if 
$$
q \in \left(\frac{1}{2}, 2\right] \cup  \bigcup_{n=1}^\infty \left[2n + \frac{1}{4}, 2n+2\right].
$$ 
\end{enumerate}
\end{thm}

\begin{figure}[htpb]
\begin{minipage}{0.5\hsize}
\begin{center}
\includegraphics[width=60mm,clip]{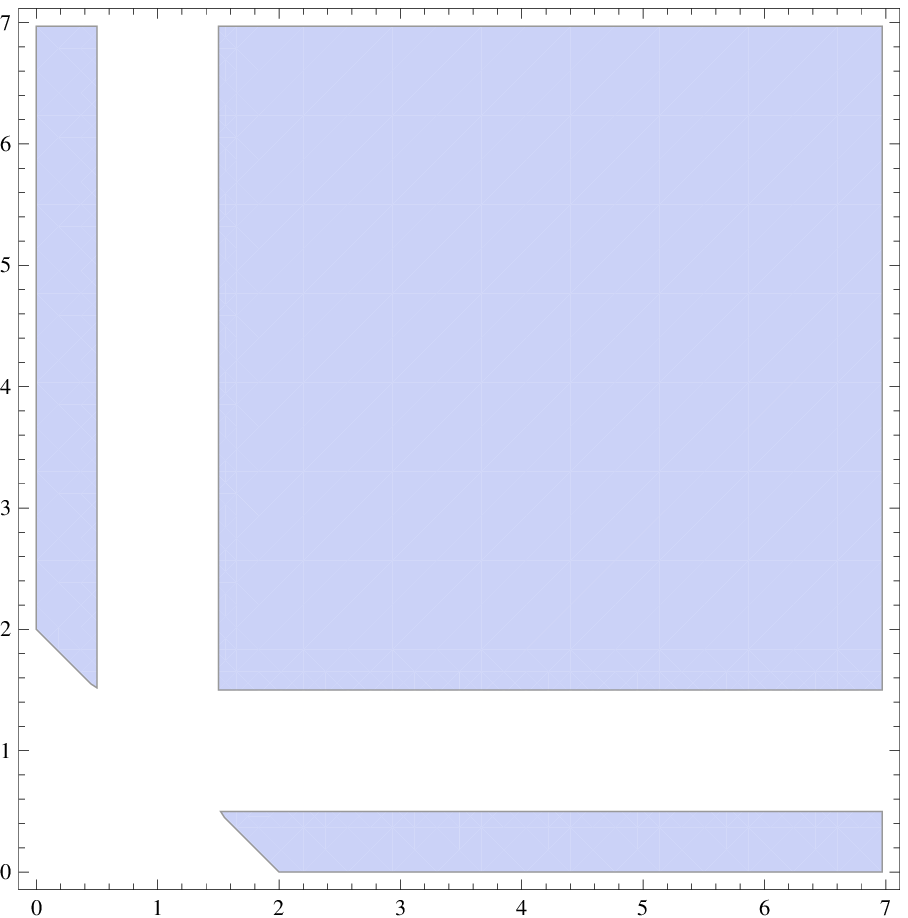}
\caption{The region for free infinite divisibility of $\bm{\beta}_{p,q}$}\label{dia1}
\end{center}
  \end{minipage}
\begin{minipage}{0.5\hsize}
\begin{center}
\includegraphics[width=60mm,clip]{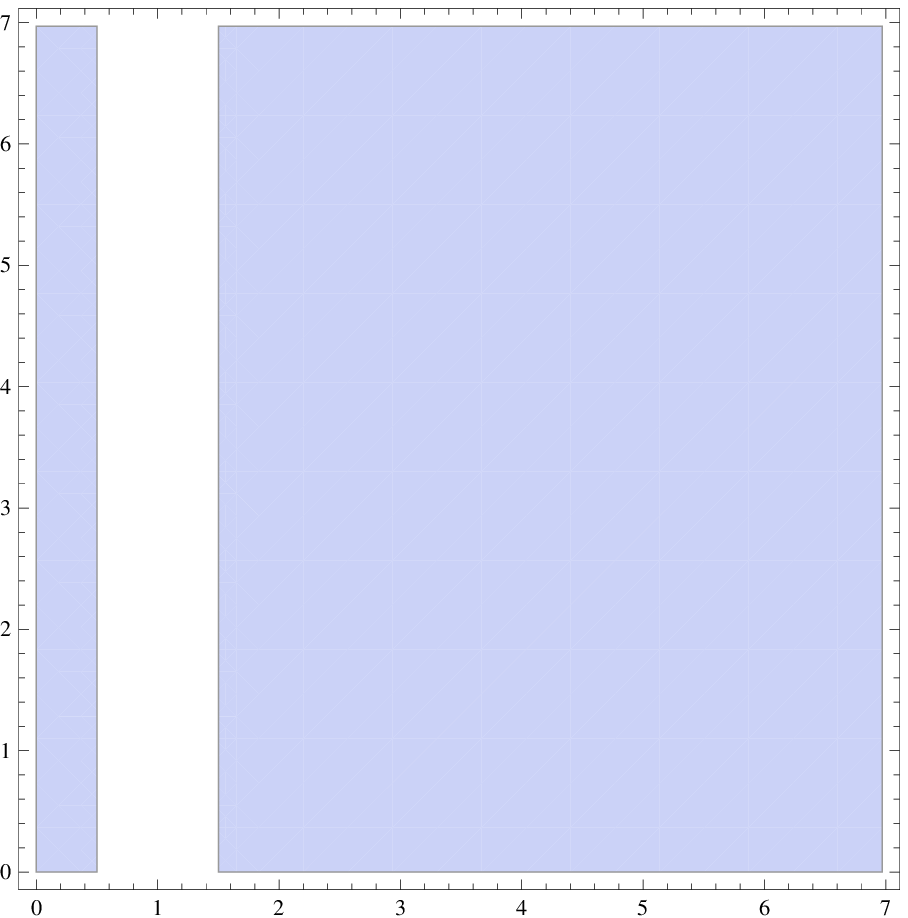}
\caption{The region for free infinite divisibility of $\bm{\beta}'_{p,q}$}\label{dia2}
\end{center}
\end{minipage}
\end{figure}

The assertions (2) and (4) follow from Theorem \ref{thm90}, a general criterion for a probability measure not to be FID. It roughly says, if a probability measure has a local density function $p(x)$ around a point $x_0$, and if $p(x)|_{(x_0-\delta,x_0+\delta)}$ is \emph{close to} the power function 
$$
c (x-x_0)^{\alpha-1}1_{(x_0,x_0+\delta)}(x)
$$ 
for some $c,\delta>0$ and $\alpha \in\mathcal{I}$, then that measure is not FID.  

Theorem \ref{thm1} has the following consequences.  

\begin{cor}\label{cor1}
\begin{enumerate}[\rm(1)]
\item The gamma distribution $\bm{\gamma}_p$ is FID if $p \in (0, \frac{1}{2}]\cup [\frac{3}{2}, \infty)$, and is not FID if $p\in \mathcal{I}$. 
\item The inverse gamma distribution $\bm{\gamma}^{-1}_p$ is FID for any $p>0$. In particular, the classical positive stable law with stability index $1/2$ is FID. 
\item The ultraspherical distribution $\mathbf{u}_p$ is FID for $p \in [1,\infty)$ and is not FID for $p\in(-\frac{1}{2},1)$. 
\end{enumerate}
\end{cor}

Corollary \ref{cor1}(1), (2) follow from limits of $\bm{\beta}_{p,q}$ and $\bm{\beta}'_{p,q}$ respectively. The assertion on non free infinite divisibility of $\bm{\gamma}_p$ is not a consequence of Theorem \ref{thm1}, but of Theorem \ref{thm90}. 
Corollary \ref{cor1}(3) for $p\in[1,\infty)$ is proved via an affine transformation of $\beta_{p+\frac{1}{2},p+\frac{1}{2}}$, and non free infinite divisibility is known in \cite[Corollary 4.1]{AP10}. This result is a positive solution of the conjecture of Arizmendi and P\'erez-Abreu \cite[Remark 4.4]{AP10}.

This paper is organized as follows. In Section \ref{sec22}, new sufficient conditions for free infinite divisibility are given, as well as the exposition of basic complex analysis in free probability. 

In Section \ref{sec2}, we express the Cauchy transforms of beta and beta prime distributions in terms of the Gauss hypergeometric function. 
We gather explicit Cauchy transforms of beta and beta prime distributions. The measures $\widetilde{\bm{\beta}}_a:=\bm{\beta}_{a, 1-a}$ and $\widetilde{\bm{\beta}}'_{a}(dx):=\bm{\beta}'_{1-a,a}(dx-1)$ are shown to satisfy
\begin{equation*}
\widetilde{\bm{\beta}}_a \mrhd \widetilde{\bm{\beta}}_b = \widetilde{\bm{\beta}}_{ab}, ~~\widetilde{\bm{\beta}}'_a \mrhd \widetilde{\bm{\beta}}'_b = \widetilde{\bm{\beta}}'_{ab}, ~~0<a,b<1, 
\end{equation*}
where $\dismrhd$ is \emph{multiplicative monotone convolution} \cite{B05}.

In Section \ref{sec3}, we prove the free infinite divisibility of beta and beta prime distributions as mentioned in Theorem \ref{thm1}. For that purpose, 
we establish first-order differential equations for the Cauchy transforms which enable us to use the sufficient conditions introduced in Section \ref{sec22}. 

In Section \ref{subsec12}, the general result for non free infinite divisibility (explained after Theorem \ref{thm1}) is shown, and it is applied to $\bm{\beta}_{p,q}, \bm{\beta}_{p,q}', \bm{\gamma}_{p}$. 

In Section \ref{st}, the free infinite divisibility for the Student distribution is proved. We also mention an easy proof of the free infinite divisibility of the Gaussian distribution in Remark \ref{gauss}.

In Section \ref{fdi}, we will provide a method for computing the free divisibility indicator of a symmetric measure and show that 
ultraspherical distributions and t-distributions mostly have free divisibility indicators equal to 1. Also the Gaussian distribution has the value 1.

\section{Free infinite divisibility}\label{sec22}
\subsection{Preliminaries}

\textbf{1.\ Tools from complex analysis.} Let $\comp^+$, $\comp^-$, $\mathbb{H}^+$ and $\mathbb{H}^-$ be the upper half-plane, lower half-plane, right half-plane and left half-plane, respectively.  
Given a Borel probability measure $\mu$ on $\real$, let $G_\mu$ be its \emph{Cauchy transform} defined by 
$$
G_\mu(z) := \int_{\real}\frac{1}{z-x}\,\mu(dx), ~~z\in \comp^+. 
$$
Its reciprocal $\displaystyle F_\mu(z):=\frac{1}{G_\mu(z)}$ is called the \emph{reciprocal Cauchy transform} of $\mu$.
When the Cauchy transform is defined in $\comp \setminus \real$, it is denoted as 
$$
\widetilde{G}_\mu(z) := \int_{\real}\frac{1}{z-x}\,\mu(dx),~~z\in\comp \setminus \real. 
$$
For a random variable $X \sim \mu$, we may write $G_X,\widetilde{G}_X$ instead of $G_\mu,\widetilde{G}_\mu$ respectively.

A measure $\mu$ can be recovered from $G_\mu$ or $\widetilde{G}_\mu$ by using the \emph{Stieltjes inversion formula}~\cite[Page 124]{A65}:  
\begin{equation}\label{eqtau}
\mu([a,b]) = -\frac{1}{\pi} \lim_{y \searrow 0}\int_{a}^b\text{Im}\, G_\mu(x+iy)dx = \frac{1}{\pi} \lim_{y \searrow 0}\int_{a}^b\text{Im}\, \widetilde{G}_\mu(x-iy)dx 
\end{equation}
for all continuity points $a, b$ of $\mu$. In particular, if the functions $f_\mu^y(x):=-\frac{1}{\pi}\text{Im}\,G_\mu(x+iy)$ converge uniformly to a continuous function $f_\mu(x)$ as $y \searrow 0$ on an interval $[a,b]$, then $\mu$ is absolutely continuous on $[a,b]$ with density $f_\mu(x)$. 
Atoms can be identified by  the formula   $\mu(\{x \})= \lim_{y \searrow 0} iy G_\mu(x+iy)$ for any $x \in \real$.

Basic properties of $G_\mu$ and $F_\mu$ are collected below; see \cite{M92} for (\ref{p2}) and \cite{BV93} for (\ref{p3}), (\ref{p4}).  
\begin{prop}\label{prop08}
\begin{enumerate}[\rm(1)]

\item\label{p1} The reciprocal Cauchy transform $F_\mu$ is an analytic map of $\comp^+$ to $\comp^+$. 

\item\label{p2} $F_\mu$ satisfies $\text{\normalfont Im}\,F_\mu(z) \geq  \text{\normalfont Im}\,z$ for $z \in \comp^+.$ If there exists $z \in \comp^+$ such that $\text{\normalfont Im}\,F_\mu(z) =  \text{\normalfont Im}\,z$, then $\mu$ must be a delta measure $\delta_a$. 

\item\label{p3} For any $\varepsilon, \lambda>0$, there exists $M>0$ such that  $|G_\mu(z)-\frac{1}{z}|\leq \frac{\varepsilon}{|z|}$ and $|F_\mu(z)-z|\leq \varepsilon |z|$ for $z \in \Gamma_{\lambda,M}$, where 
$$
\Gamma_{\lambda, M}:=\{z \in \mathbb{C}^+: \text{\normalfont Im}\,z >M,~ |\text{\normalfont Re}\,z| <\lambda \text{\normalfont Im}\,z\}. 
$$  
\item\label{p4} For any $0<\varepsilon< \lambda$, there exists $M>0$ such that $F_\mu$ is univalent in $\Gamma_{\lambda, M}$ and $F_\mu(\Gamma_{\lambda, M}) \supset \Gamma_{\lambda-\varepsilon, (1+\varepsilon)M}$, and so the inverse map $F_\mu^{-1}: \Gamma_{\lambda-\varepsilon, (1+\varepsilon)M}\to \comp^+$ exists such that $F_\mu \circ F_\mu^{-1}=\text{\normalfont Id}$ in $\Gamma_{\lambda-\varepsilon, (1+\varepsilon)M}$. 
\item\label{p5} If $\mu$ is symmetric, then $\text{\normalfont Im}\, G_\mu(x+iy)=\text{\normalfont Im}\, G_\mu(-x+iy)$ and $\text{\normalfont Re}\, G_\mu(x+iy)=-\text{\normalfont Re}\, G_\mu(-x+iy)$ for $x \in \real,~y>0$. In particular,   
$G_\mu(i(0,\infty)) \subset i(-\infty,0)$.  
\end{enumerate}
\end{prop}

In addition, the following property is used in Section \ref{st}. 
\begin{lem}[\cite{BH}, Lemma 3.2]\label{lem03}
If a probability measure $\mu$ has a density $p(x)$ such that $p(x) = p(-x)$, $p'(x) \leq 0$ for a.e.\ $x > 0$ and $\lim_{x\to \infty}  p(x)\log x =0$, then $\text{\normalfont Re}\, G_\mu(x+yi)>0$ for $x,y>0$. 
 \end{lem}
Note that some symmetric probability measures do not satisfy the property $\text{\normalfont Re}\, G_\mu(x+yi)>0$ for $x,y>0$. The Bernoulli law $\mathbf{b}=\frac{1}{2}(\delta_{-1}+\delta_1)$  has the Cauchy transform $G_\mathbf{b}(z)=\frac{z}{z^2-1}$ and so $G_\mathbf{b}(\frac{1}{2}e^{i \pi /4}) = \frac{\sqrt{2}}{17}(-3-4i).$ 

\vspace{15pt}

\textbf{2.\ Free convolution and freely infinitely divisible distributions.} 
If $X_{1},X_{2}$ are free random variables following probability distributions $\mu_1, \mu_2$ respectively, then the probability distribution of $X_1+X_2$ is denoted by $\mu_1 \boxplus \mu_2$ and is called the \emph{free additive convolution} of $\mu_1$ and $\mu_2$. Free additive convolution is characterized as follows \cite{BV93}. 
From Proposition \ref{prop08}(\ref{p4}), for any $\lambda>0$, there is $M>0$ such that the right compositional inverse map $F^{-1}_\mu$ exists in $\Gamma_{\lambda, M}$. 
Let $\phi_\mu(z)$ be the \emph{Voiculescu transform} of $\mu$ defined by 
\begin{equation}\label{eq44}
\phi_\mu(z) := F_\mu^{-1}(z)-z,~~z\in\Gamma_{\lambda,M}.  
\end{equation}
The free convolution $\mu \boxplus \nu$ is the unique probability measure such that 
$$
\phi_{\mu \boxplus \nu}(z) = \phi_\mu (z) + \phi_\nu(z)
$$
 in a common domain of the form $\Gamma_{\lambda', M'}$. 

Free convolution associates a basic class of probability measures, called \emph{freely infinitely divisible distributions} introduced in \cite{V86} for compactly supported probability measures and in \cite{BV93} for all probability measures. 
\begin{defi} A probability measure $\mu$ on $\real$ is said to be \textit{freely infinitely divisible} (or FID for short)  if for each $n\in \{1,2,3,\cdots\}$ there exists a probability measure $\mu_n$ such that
$$
\mu =\mu_n^{\boxplus n}:=\underbrace{\mu_n \boxplus \cdots \boxplus \mu_n.}_{\text{$n$ times}}
$$ 
The set of FID distributions is closed with respect to weak convergence \cite[Theorem 5.13]{BT06}.  FID distributions appear as the limits of infinitesimal arrays as in classical probability theory; see \cite{CG08}. 
\end{defi}
FID distributions are  characterized in terms of a complex analytic property of Voiculescu transforms. 
\begin{thm}[\cite{BV93}] \label{thm0}
For a probability measure $\mu$ on $\real$, the following are equivalent. 
\begin{enumerate}[\rm(1)] 
\item $\mu$ is FID. 
\item $-\phi_\mu$ extends to a Pick function, i.e.\ an analytic map from $\comp^+$ into $\comp^+ \cup \real$. 
%\item There exist $\gamma \in \real$ and a nonnegative finite measure $\tau$ such that   
%\begin{equation}\label{eq010}
%\phi_\mu(z)= \gamma + \int_{\real}\frac{1+xz}{z-x}\,\tau(dx),~~z \in \comp^+.  
%\end{equation}
\item For any $t>0$, there exists a probability measure $\mu^{\boxplus t}$ with the property $\phi_{\mu^{\boxplus t}}(z) = t\phi_\mu(z)$ in some $\Gamma_{\lambda,M}$. 
\end{enumerate}
\end{thm}
Note that Pick functions are also crucial in the characterization of \emph{generalized gamma convolutions} (GGCs) in classical probability \cite{B92}.

\subsection{Sufficient conditions for free infinite divisibility}
When the Voiculescu transform does not have an explicit expression, the conditions in Theorem \ref{thm0} are difficult to check. In such a case, a subclass $\mathcal{UI}$ of FID measures has been exploited in the literature \cite{BBLS11,ABBL10,AB,AHa,AHb,BH,H}. We also introduce a variant of it. 
\begin{defi}\label{def1} \begin{enumerate}[\rm(1)]
\item A probability measure $\mu$ is said to be in class $\mathcal{UI}$ if  $F_\mu^{-1}$, defined in some $\Gamma_{\lambda,M}$, analytically extends to a univalent map in $\mathbb{C}^+$. $\mu \in \mathcal{UI}$ if and only if there is an open set $\Omega \subset \comp$, $\Omega \cap \Gamma_{\lambda,M} \neq \emptyset$ such that $F_\mu$ extends to an analytic bijection of $\Omega$ onto $\comp^+.$
\item A symmetric probability measure $\mu$ is said to be in class $\mathcal{UI}_s$ if: (a) there is $c \leq 0$ such that $F_\mu$ extends to a univalent map around $i(c,\infty)$ and maps $i(c,\infty)$ onto $i(0,\infty)$; (b) there is an open set $\widetilde{\Omega} \subset \comp^-\cup \mathbb{H}^+$ such that $\widetilde{\Omega} \cap \Gamma_{\lambda,M} \neq \emptyset$, where $ \Gamma_{\lambda,M}$ is the cone defined in the paragraph prior to (\ref{eq44}),  and that $F_\mu$ extends to an analytic bijection of $\widetilde{\Omega}$ onto $\comp^+ \cap \mathbb{H}^+$. 
\end{enumerate}
\end{defi}
\begin{rem}\label{rem56}
In \cite{AHb} we required $F_\mu$ to be univalent in $\comp^+$ in the definition of $\mu \in \mathcal{UI}$, but this automatically follows. If $F^{-1}_\mu$ is analytic in $\comp^+$, then $F^{-1}_\mu \circ F_\mu (z) =z$ for $z \in \comp^+$ by analyticity, so that $F_\mu$ should be univalent in $\comp^+$.    
\end{rem}
\begin{lem}\label{lem112}
If $\mu \in \mathcal{UI}$ or $\mu \in \mathcal{UI}_s$, then $\mu$ is FID. 
\end{lem}
\begin{proof}
The proof for $\mathcal{UI}$ is found in \cite{AHb,BBLS11}. Assume $\mu \in \mathcal{UI}_s$. We are able to define 
\begin{equation}\label{eq59}
F^{-1}_\mu(z):= 
\begin{cases}
F_\mu|_{\widetilde{\Omega}}^{-1}(z), & z \in \comp^+ \cap \mathbb{H}^+, \\
F_\mu|_{i(c,\infty)}^{-1}(z),&z \in i(0,\infty), \\
F_\mu|_{\widetilde{\Omega}^\ast}^{-1}(z), & z \in \comp^+ \cap \mathbb{H}^-,  \\
\end{cases}
\end{equation}
where $\widetilde{\Omega}^\ast:=\{-x+iy: x+iy \in \widetilde{\Omega} \}$ and $F_\mu|_A$ is the restriction of $F_\mu$ to a set $A$. This is well defined because each of $\widetilde{\Omega}$, $i(c,\infty)$ and $\widetilde{\Omega}^\ast$ has nonempty intersection with $\Gamma_{\lambda,M}$, and so each of $F_\mu|_{\widetilde{\Omega}} ^{-1}(z), F_\mu|_{i(c,\infty)}^{-1}(z)$ and $F_\mu|_{\widetilde{\Omega}^\ast}^{-1}(z)$ coincides with the original inverse (\ref{eq44}) in the common domain. Note that, as explained in Remark \ref{rem56}, $F_\mu$ is univalent in $\comp^+$. 

The remaining proof is similar to the case $\mu \in \mathcal{UI}$. Take $z \in \comp^+ \cap \mathbb{H}^+$. If $z \in F_\mu(\comp^+)$, then taking the preimage $w \in \comp^+$ of $z$ and we see 
$\text{Im}\,\phi_\mu(z)=\text{Im}(F_\mu^{-1}(z) -z) =\text{Im}\,w - \text{Im}\,F_\mu(w)$ is not positive. If $z \notin F_\mu(\comp^+)$, the preimage $w \in \widetilde{\Omega}$ must be in $\comp^- \cup \real$ since $F_\mu|_{\comp^+}$ is univalent, so that  $\text{Im}\,\phi_\mu(z)=\text{Im}(w-z) \leq 0$. 
Therefore, $\phi_\mu$ maps $\comp^+$ into $\comp^- \cup \real$. The other two  cases $z \in i(0,\infty)$ and $z \in \comp^+ \cap \mathbb{H}^-$ are similar. 
\end{proof}
\begin{rem}
 It holds that $\mathcal{UI} \cap\{\mu: \text{symmetric}\} \subset \mathcal{UI}_s$.  For general $\mu \in \mathcal{UI}_s$, the map $F_\mu^{-1}$ may not be univalent in $\comp^+$, but 
 it is 2-valent, i.e., for each $z \in \comp^+$, $\sharp\{w\in\comp^+: F_\mu^{-1}(w)=F_\mu^{-1}(z) \} = 1$ or $2$. 
\end{rem}

The following conditions on a Cauchy transform are quite useful to prove the free infinite divisibility of a probability measure. 
\begin{enumerate}[(A)]
\item\label{A} There is a connected open set $\comp^+ \subset\mathcal{D} \subset \comp$ such that: 
\begin{enumerate}[\rm(\text{A}1)]
\item\label{merA} $G_\mu$ extends to a meromorphic function in $\mathcal{D}$; 

\item\label{diffA} If $G_\mu(z) \in \comp^-$ and $z \in \mathcal{D}$, then $G_\mu'(z)\neq 0$;   

\item\label{barrierA} If a sequence $(z_n)_{n \geq 1} \subset \mathcal{D}$ converges to a point of $\partial \mathcal{D} \cup \{\infty\}$, then the limit
$\lim_{n\to \infty} G_\mu(z_n)$ exists in $\comp^{+} \cup \real  \cup \{\infty\}$.    
\end{enumerate}
\end{enumerate}
Condition (\ref{diffA}) is useful to define an inverse map $F_\mu^{-1}$ in $\comp^+$. This condition was crucial in the proof of free infinite divisibility of the normal distribution \cite{BBLS11}. 
Condition (\ref{barrierA}) is used to show the map $F_\mu^{-1}$ is univalent in $\comp^+$. (\ref{barrierA}) is important  as well as (\ref{merA}) and (\ref{diffA}) because the exponential distribution satisfies (\ref{merA}) and (\ref{diffA}) for $\mathcal{D}=\comp\setminus(-\infty,0]$, but does not satisfy (\ref{barrierA}). 
It is known that the exponential distribution is not FID, see Section \ref{Hankel}.

For symmetric distributions, the following variant can be more useful. 
\begin{enumerate}[(A)]
\setcounter{enumi}{1}
\item\label{B} There is $c \leq 0$ such that $G_\mu$ extends to a univalent map around $i(c,\infty)$ and maps $i(c,\infty)$ onto $i(-\infty,0)$. 
Moreover, there is a connected open set $\comp^+ \cap \mathbb{H}^+ \subset \mathcal{D} \subset \comp^- \cup \mathbb{H}^+$ such that:  
\begin{enumerate}[\rm(\text{B}1)] 
\item\label{merB} $G_\mu$ extends to a meromorphic function in $\mathcal{D}$; 
\item\label{diffB} If $G_\mu(z) \in \comp^-$ and $z \in \mathcal{D}$, then $G_\mu'(z)\neq 0$;   
\item\label{barrierB} If a  sequence $(z_n)_{n \geq 1} \subset \mathcal{D}$ converges to a point of $\partial \mathcal{D}  \cup \{\infty\}$, then the limit  $\lim_{n\to \infty} G_{\mu}(z_n)$ exists in $\overline{\mathbb{H}^- \cup \comp^+}  \cup \{\infty\}$. 
\end{enumerate}
\end{enumerate}

\begin{prop}\label{prop91}\begin{enumerate}[\rm(1)]
\item
If the Cauchy transform $G_\mu$ of a probability measure $\mu$ satisfies (\ref{A}),  then $\mu \in \mathcal{UI}$. 
\item If the Cauchy transform of a symmetric probability measure $\mu$ satisfies (\ref{B}),  then $\mu \in \mathcal{UI}_s$. If, moreover, the domain $\mathcal{D}$ can be taken as a subset of $\mathbb{H}^+$, then $\mu \in \mathcal{UI}$. 
\end{enumerate}
\end{prop}
\begin{proof}
(1)\,\,\,Let $c_t \subset \comp^+$ be the curve defined by  
$$
c_t:= \{x+yi: t y= |x|+1,~x \in \real \},~~t>0. 
$$
Note that $\bigcup_{t>0}c_t = \comp^+$. 
From Proposition \ref{prop08}(\ref{p2}), for each $t>0$, if we take a large $R>0$,  there exists a simple curve $\gamma_t^R$ such that 
$F_\mu(\gamma_t^R) = c_t \cap\{z \in \comp^+: \text{Re}\,z >R\}$ and $F_\mu$ maps a neighborhood of $\gamma_t^R$ onto a neighborhood of $c_t \cap\{z \in \comp^+: \text{Re}\,z >R\}$ bijectively. Take a sequence $z_n \in \gamma_t^R$ converging to the edge of $\gamma_t^R$ which we denote by $z^R$, then $F_\mu(z_n)$ converges to $F_\mu(z^R) \in c_t$. Condition (\ref{diffA}) implies that $F_\mu'(z^R) \neq 0$, so that there is an open neighborhood $V^R$ of $z^R$ such that $F_\mu$ maps $V^R$ bijectively onto a neighborhood of $F_\mu(z^R)$. Hence we obtain a curve $\gamma_t^{R-\varepsilon} \supset \gamma_t^R$ for some $\varepsilon>0$ such that $F_\mu(\gamma_t^{R-\varepsilon})= c_t \cap\{z \in \comp^+: \text{Re}\,z >R-\varepsilon\}$. Repeating this argument, we can prolong $\gamma_t^R$ to obtain a maximal curve $\gamma_t \subset \mathcal{D}$ such that $F_\mu$ maps $\gamma_t$ into $c_t$. We show that $F_\mu(\gamma_t)=c_t$, and for this purpose we assume $F_\mu(\gamma_t) $ is a proper subset of $c_t$. 
 Let $x_0:= \inf\{x \in \real: x+\frac{i}{t}(|x|+1) \in F_\mu(\gamma_t)\} \in \real$ and $z_0:= x_0 + \frac{i}{t}(|x_0|+1) \in c_t.$ For a point $z\in c_t, \text{Re}\, z > x_0$, let $w \in \gamma_t$ denote the preimage of $z$. 
 The following cases are possible: 
\begin{enumerate}[\rm(i)] 
\item\label{case1} When $z$ converges to $z_0$, the preimages $w$ have an accumulative point $w_0$ in $\mathcal{D}$; 

\item\label{case2}  When $z$ converges to $z_0$, the preimages $w$ have an accumulative point $w_1$ in $\partial \mathcal{D} \cup \{\infty \}$.  
\end{enumerate}
In the case (\ref{case1}), we can still extend the curve $\gamma_t$ more because of condition (\ref{diffA}) and the obvious fact $F_\mu(w_0)=z_0$; a contradiction to the maximality of $\gamma_t$. The point $w_0$ might be a pole of $F_\mu$, but in that case $z_0$ has to be infinity, which is again a contradiction. In the case (\ref{case2}), condition (\ref{barrierA})  implies $z_0= \lim_{w\to w_1, w\in \mathcal{D}}F_\mu(w) \in \comp^-\cup \real \cup\{\infty\}$, while $z_0 \in c_t \subset \comp^+$, again a contradiction. Thus we conclude that $F_\mu(\gamma_t)=c_t$. Note that  $F_\mu$ maps an open neighborhood $U_t$ of $\gamma_t$ onto a neighborhood of $c_t$ bijectively. Hence the set $\Omega:=\bigcup_{t>0}U_t \subset \mathcal{D}$ is open and $F_\mu$ maps $\Omega$ bijectively onto $\comp^+$. This implies that $F_\mu|_{\Omega}^{-1}$ exists as a univalent map in $\comp^+$. Since $\Omega$ has intersection with the original domain $\Gamma_{\lambda,M}$ of the right inverse $F_\mu^{-1}$, the map $F|_{\Omega}^{-1}$ extends $F_\mu^{-1}$ analytically, and hence $\mu \in \mathcal{UI}$.

(2)\,\,\, The proof is quite similar. Let $\widetilde{c}_t:= c_t \cap \mathbb{H}^+$. One can prolong the above $\gamma_t^R$, to obtain $\widetilde{\gamma}_t \subset \mathcal{D}$ such that $F_\mu(\widetilde{\gamma}_t) =\widetilde{c}_t.$ Denoting by $\widetilde{U}_t$ an open neighborhood of $\widetilde{\gamma}_t$ where $F_\mu$ is univalent, $F_\mu$ maps $\widetilde{\Omega}:=\bigcup_{t>0}\widetilde{U}_t \subset \mathcal{D}$ bijectively onto $\comp^+ \cap \mathbb{H}^+$, so that $\mu\in\mathcal{UI}_s$. Moreover, if $\mathcal{D} \subset \mathbb{H}^+$, then the map $F_\mu|_{\widetilde{\Omega}}^{-1}$ defined in (\ref{eq59}) is univalent in $\comp^+$. 
\end{proof}
\begin{rem}
Condition (\ref{diffA}) enables us to construct the curve $\gamma_t$, but $\gamma_t$  can enter another Riemannian sheet of $F_\mu$ beyond $\partial \mathcal{D}$. Condition (\ref{barrierA}) becomes a ``barrier'' which prevents such a phenomenon. If $F_\mu$ is a rational function in $\comp$ as in the case of Student distributions for $q$ integers, 
there is no other branch of $F_\mu$ and we can take $\mathcal{D}=\comp$ and condition (\ref{barrierA}) is easily verified.  This will give a simple proof of free infinite divisibility of Gaussian (see Section \ref{st}).  
\end{rem}

\section{Cauchy transforms of beta, beta prime and Student t-distributions}\label{sec2}
Let $F(a,b;c; z)$ be the \emph{Gauss hypergeometric series}:  
$$
F(a,b;c; z) = \sum_{n=0}^\infty \frac{(a)_n (b)_n}{( c)_n}\frac{z^n}{n!},~~c \notin \{0, -1,-2,-3,\cdots\}
$$
with the conventional notation $(a)_n:=a(a+1)\cdots(a+n-1)$, $(a)_0:=1$. This series is absolutely convergent for $|z|<1$. 
There is an integral representation 
\begin{equation}\label{int}
F(a,b;c; z)=\frac{1}{B(c-b,b)}\int_{0}^1 x^{b-1}(1-x)^{c-b-1}(1-zx)^{-a}dx,  \quad \text{Re}(c ) >\text{Re}(b )>0,  
\end{equation}
which continues $F(a,b;c; z)$ analytically to $\comp \setminus [1, \infty)$. The normalizing constant $B(p,q)$ is the beta function which is related to the gamma function as $B(p,q)= \frac{\Gamma(p{})\Gamma({}q)}{\Gamma(p+q)}.$  

We note some formulas required in this paper \cite[Chapter 15]{AS70}. 
\begin{align}
c(1-z) F(&a,b;c; z) - cF(a-1, b;c; z) + (c-b)zF(a,b;c+1; z)=0, \label{15.2.20}\\   %15.2.20
F(a,b;c; z) &= (1-z)^{c-a-b} F(c-a,c-b;c; z)~~~~~(|\arg(1-z)|<\pi),\label{15.3.3}%15.3.3
\end{align}
\vspace{-22pt}
\begin{align}
F(a,b;c; z) &= \frac{\Gamma({}c)\Gamma(b-a)}{\Gamma(b)\Gamma(c-a)}(-z)^{-a}F\left(a,1-c+a;1-b+a; \frac{1}{z}\right)  \notag \\
&~~~+\frac{\Gamma({}c)\Gamma(a-b)}{\Gamma(a)\Gamma(c-b)}(-z)^{-b}F\left(b,1-c+b;1-a+b; \frac{1}{z}\right) \notag \\ 
&~~~~~~~~~~~~~~~~~~~~~~~~~~~~~~~~~~~~~~~~~~~~~~ (b-a \notin\mathbb{Z},~|\arg (-z)|<\pi), \label{15.3.7} %15.3.7
\end{align}
\vspace{-22pt}
\begin{align}
F(a,b;c; z) &= \frac{\Gamma({}c)\Gamma(c-a-b)}{\Gamma(c-a)\Gamma(c-b)}z^{-a}F\left(a,a-c+1;a+b-c+1; 1-\frac{1}{z}\right)  \notag \\
&~~~+\frac{\Gamma({}c)\Gamma(a+b-c)}{\Gamma(a)\Gamma(b)}(1-z)^{c-a-b}z^{a-c}F\left(c-a,1-a;c-a-b+1;1- \frac{1}{z}\right) \notag \\ 
&~~~~~~~~~~~~~~~~~~~~~~~~~~~~~ (a+b-c\notin\mathbb{Z},~|\arg z|, |\arg(1-z)|<\pi). \label{15.3.9} %15.3.7
\end{align}
The branch of every $z^{p}$ is the principal value. When $b-a \in \mathbb{Z}$, all terms in (\ref{15.3.7}) diverge, but an alternative formula is available \cite[15.3.14]{AS70}. 
The formula (\ref{15.3.7}), however, is sufficient for our purpose. Similarly, we do not use an alternative formula for (\ref{15.3.9}). 

The following properties are useful for calculating the Cauchy transforms of beta prime and t-distributions. 
\begin{lem} \label{lem2}
\begin{enumerate}[\rm(1)]
\item\label{inverse} Let $X$ be a $\real$-valued random variable such that $X\neq 0$ a.s. Then 
$$
\widetilde{G}_{1/X}(z) = \frac{1}{z}-\frac{1}{z^2}\widetilde{G}_X\left(\frac{1}{z}\right), ~~z \in \comp\setminus \real. 
$$

\item\label{affine} Let $X$ be a $\mathbb{R}$-valued random variable. Then, for $a \neq 0$ and $b \in \real$,  
$$
\widetilde{G}_{a X +b}(z) = \frac{1}{a}\widetilde{G}_{X}\left(\frac{z-b}{a}\right),~~z \in \comp\setminus\real. 
$$

\item\label{square} If $X$ is a $\real$-valued symmetric random variable, then 
$$
G_{X}(z) = z\widetilde{G}_{X^2}(z^2),~~z\in \comp^+. 
$$
\end{enumerate}
\end{lem}
\begin{proof}
Let $\mu$ be the distribution of $X$. \\

(1)\,\,\,
$
\displaystyle 
\widetilde{G}_{1/X}(z)= \int_{\real} \frac{1}{z-1/x}\, \mu(dx) = \frac{1}{z}\int_{\real} \frac{x-1/z+1/z}{x-1/z}\, \mu(dx) \\[5pt]
~~~~~~~~~~~~~~~~~~~~~\,= \frac{1}{z} - \frac{1}{z^2}\int_{\real} \frac{1}{1/z-x}\, \mu(dx) =\frac{1}{z}-\frac{1}{z^2}\widetilde{G}_X\left(\frac{1}{z}\right).
$
\\

(2) is easy to prove. \\

(3)\,\,\,
$
\displaystyle 
G_{X}(z)= \int_{0}^\infty \frac{1}{z-x}\, \mu(dx) +  \int_{-\infty}^0 \frac{1}{z-x}\, \mu(dx)= \int_{0}^\infty\left(\frac{1}{z-x}+ \frac{1}{z+x}\right)\mu(dx) \\[5pt]
~~~~~~~~~~~~~~~~~~~= \int_{0}^\infty \frac{2z}{z^2-x^2}\, \mu(dx) =z\widetilde{G}_{X^2}(z^2). 
$
\end{proof}
Now we are going to compute the Cauchy transforms of $\bm{\beta}_{p,q}, \bm{\beta}'_{p,q}$ and $\mathbf{t}_q$ in terms of hypergeometric series. 
\begin{prop}\label{prop113}
\begin{enumerate}[\rm(1)]
\item $
\displaystyle G_{\bm{\beta}_{p,q}}(z)=\frac{1}{z}  F(1,p;p+q; z^{-1}) 
$
for $z \in \comp^+$.  

\item 
$\displaystyle
\widetilde{G}_{\bm{\beta}'_{p,q}} (z) 
= \frac{1}{z+1}+\frac{1}{(z+1)^2}\widetilde{G}_{\bm{\beta}_{p,q}}\left(\frac{z}{z+1}\right)\\[5pt]
\,\,\,\,\,\,\,\,\,\,\,\,\,\,\,\,\,\,\,\,
 =\frac{q}{(p+q)z} F\left(1,p;1+p+q; 1+\frac{1}{z}\right), ~~ z\in \comp\setminus[0,\infty). 
 $

\item $\displaystyle G_{\mathbf{t}_q}(z) =\frac{q-\frac{1}{2}}{q}\frac{1}{z}F\left(1,\frac{1}{2};1+q; 1+\frac{1}{z^2}\right)$ for $z \in \comp^+$. 
\end{enumerate}
\end{prop}
\begin{proof}
(1)\,\,\, This is easy from the integral representation (\ref{int}) of the hypergeometric series.

(2)\,\,\, If $X \sim\bm{\beta}_{p,q}$, then $\frac{X}{1-X}\sim\bm{\beta}'_{p,q}$, so that $\displaystyle G_{\bm{\beta}'_{p,q}}(z)=\frac{1}{z+1}+\frac{1}{(z+1)^2}G_{\bm{\beta}_{p,q}}\left(\frac{z}{z+1}\right)$ from Lemma \ref{lem2}(\ref{inverse}), (\ref{affine}). 
 Hence 
\begin{equation}\label{eq0021}
G_{\bm{\beta}'_{p,q}}(z) = \frac{1}{1+z} + \frac{1}{z(1+z)}  F\left(1,p;p+q; 1+\frac{1}{z}\right). 
\end{equation}
The formula (\ref{15.2.20}) can increase the parameter $p+q$ by 1: 
$$
F\left(1,p;p+q; 1+\frac{1}{z}\right)= \frac{q}{p+q}(1+z)F\left(1,p;1+p+q; 1+\frac{1}{z}\right) -z. 
$$
This, together with (\ref{eq0021}), leads to the conclusion. 

(3)\,\,\,We can use Lemma \ref{lem2}(\ref{square}) because $X \sim \mathbf{t}_q$ implies $X^2 \sim \bm{\beta}'_{1/2, q-1/2}$. 
\end{proof}

\begin{exa}
Some hypergeometric functions and hence the corresponding $G_{\bm{\beta}_{p,q}}, G_{\bm{\beta}'_{p,q}}$ have explicit forms. Examples are presented here. 
\begin{enumerate}[\rm(1)]
\item $F(1,a;1;z) = (1-z)^{-a}$ from the formula (\ref{15.3.3}), and hence 
$$
G_{\bm{\beta}_{a,1-a}}(z) = \frac{1}{z}\left(1-\frac{1}{z}\right)^{-a},~~ 0<a<1,~~|\arg (-z)| <\pi. 
$$
\item From (\ref{15.2.20}), we have $(1-z)F(1,a;1;z)-1+(1-a)zF(1,a;2;z)=0$, and hence $zF(1,a;2;z)=\frac{1-(1-z)^{1-a}}{1-a}.$ The Cauchy transform of $\bm{\beta}_{1-a, 1+a}$ is given by 
$$
G_{\bm{\beta}_{1-a,1+a}}(z) = \frac{1}{a}\left(1-\left(1-\frac{1}{z}\right)^a\right),~~ -1<a<1, ~~|\arg (-z)| <\pi. 
$$
\item Similarly, we can calculate $zF(1,a;3;z) =\frac{2\left( (2-a)z-1+(1-z)^{2-a}\right)}{(2-a)(1-a)z}$ and hence 
$$
G_{\bm{\beta}_{2-a,1+a}}(z)=\frac{2\left(a-z+z(1-\frac{1}{z})^a\right)}{a(a-1)},~~ -1<a<2,~~|\arg (-z)| <\pi.
$$
\end{enumerate}
For beta prime distributions, the formula $G_{\bm{\beta}'_{q,p}}(z)=\frac{1}{z+1}-\frac{1}{(z+1)^2}G_{\bm{\beta}_{p,q}}(\frac{1}{z+1})$ holds because of Lemma \ref{lem2} and of the fact that $X \sim \bm{\beta}_{p,q}$ implies $\frac{1}{X}-1\sim \bm{\beta}'_{q,p}$. Explicit formulas are therefore easy to calculate.  
\begin{enumerate}[\rm(1)]
\setcounter{enumi}{3}
\item $\displaystyle G_{\bm{\beta}'_{1-a,a}}(z) = \frac{1-(-z)^{-a}}{1+z}$, ~~$0<a<1$, ~~$|\arg (-z)| <\pi.$ 

\item $\displaystyle G_{\bm{\beta}'_{1+a,1-a}}(z) = \frac{1}{1+z} -\frac{1-(-z)^a}{a(1+z)^2}$,~~ $-1<a<1$,~~$|\arg (-z)| <\pi.$

\item $\displaystyle G_{\bm{\beta}'_{1+a,2-a}}(z)=\frac{1}{1+z}-\frac{2\left(az +a-1 +(-z)^a\right)}{a(a-1)(1+z)^3},$ ~~$-1<a<2$,~~$|\arg (-z)| <\pi.$
\end{enumerate}

\textbf{Note.} The measure $\bm{\beta}_{1-a,1+a}$ appeared in \cite{AHb} and $\bm{\beta}_{a,1-a}$ appeared in \cite{M10}. 
Demni computed explicitly \emph{generalized Cauchy-Stieltjes transforms} of beta distributions \cite{D09}. 
\end{exa}

\begin{rem}  
There is a relation (\ref{eq001}) involving monotone convolution and beta distributions. We will find more, replacing monotone convolution by multiplicative monotone convolution. 
The \emph{multiplicative monotone convolution} $\mu \dismrhd \nu$ of probability measures $\mu,\nu$ on $[0,\infty)$ is the distribution of $\sqrt{X}Y\sqrt{X}$, where $X,Y$ are positive random variables, respectively following the distributions $\mu,\nu$, and
$X-1$ and $Y-1$ are monotonically independent \cite{B05,F09}. It is characterized by 
$$
\eta_{\mu \submrhd \nu}(z)=\eta_{\mu}(\eta_{\nu}(z)),~~~z\in (-\infty,0), 
$$
where 
$\eta_\mu(z):= 1-z F_\mu(\frac{1}{z})$ is called the $\eta$-transform. 

Let $\widetilde{\bm{\beta}}_{a}:=\bm{\beta}_{a,1-a}$ and $\widetilde{\bm{\beta}}'_{a}(dx):=\bm{\beta}'_{1-a,a}(dx-1)$. Then $G_{\widetilde{\bm{\beta}}'_{a}}(z) =  \frac{1-(1-z)^{-a}}{z}$ and  
$$
\eta_{\widetilde{\bm{\beta}}_{a}}(z)=1-(1-z)^a,~~\eta_{\widetilde{\bm{\beta}}'_{a}}(z)=\frac{(-z)^a}{(-z)^a-(1-z)^a}, ~~z<0,  
$$
which entail 
$$
\eta_{\widetilde{\bm{\beta}}_{a}} \circ \eta_{\widetilde{\bm{\beta}}_{b}} = \eta_{\widetilde{\bm{\beta}}_{ab}},~~\eta_{\widetilde{\bm{\beta}}'_{a}}\circ \eta_{\widetilde{\bm{\beta}}'_{b}} =\eta_{\widetilde{\bm{\beta}}'_{ab}},
$$
or equivalently 
$$
\widetilde{\bm{\beta}}_{a} \dismrhd \widetilde{\bm{\beta}}_{b} =\widetilde{\bm{\beta}}_{ab},~~\widetilde{\bm{\beta}}'_{a} \dismrhd \widetilde{\bm{\beta}}'_{b}=\widetilde{\bm{\beta}}'_{ab}
$$
for $0<a,b<1$. Hence the measures $\widehat{\bm{\beta}}_t:=\bm{\beta}_{e^{-t},1-e^{-t}}$ and $\widehat{\bm{\beta}}'_t(dx):=\bm{\beta}'_{1-e^{-t},e^{-t}}(dx-1)$ both form $\dismrhd$-convolution semigroups 
with initial measure $\delta_1$ at $t=0.$ 
\end{rem}

\section{Free infinite divisibility for beta and beta prime distributions}\label{sec3}

In order to find a good domain $\mathcal{D}$ such that condition (\ref{A}) holds, the following alternative condition is useful.  
\begin{enumerate}[(A)]
\setcounter{enumi}{2}
\item\label{C} There is a connected open set $\comp^+ \subset\mathcal{E} \subset \comp$ such that: 
\begin{enumerate}[\rm(\text{C}1)]
\item\label{merC} $G_\mu$ extends to an \emph{analytic function} in $\mathcal{E}$; 
\item\label{diffC} If $G_\mu(z) \in \real$ and $z \in \mathcal{E}$, then $G_\mu'(z)\neq 0$.  
\end{enumerate}
\end{enumerate} 
The usage of this condition becomes clear in Theorem \ref{thm11}, \ref{thm12}. Remark \ref{rem09} also explains why this condition is important.  

We are going to prove conditions (\ref{A}) and (\ref{C}) for beta and beta prime distributions. 
The following result shows conditions (\ref{merA}) and (\ref{merC}), and moreover explicit formulas of the analytic continuation of Cauchy transforms. 
\begin{prop}\label{prop78}\begin{enumerate}[\rm(1)]
\item The Cauchy transform $G_{\bm{\beta}_{p,q}}$ analytically extends to $\mathcal{D}^b=\mathcal{E}^b:=\comp \setminus\left( (-\infty,0] \cup [1,\infty)\right)$. Denoting the analytic continuation by the same symbol  $G_{\bm{\beta}_{p,q}}$, we obtain
\begin{align}
&G_{\bm{\beta}_{p,q}}(z)
=  \widetilde{G}_{\bm{\beta}_{p,q}}(z) - \frac{2\pi i }{B(p,q)}z^{p-1}(1-z)^{q-1}, ~~z \in \comp^{-}. \label{eqbeta} 
\end{align}

\item\label{betapp} The Cauchy transform $G_{\bm{\beta}'_{p,q}}$ analytically extends to $\mathcal{D}^{bp}=\mathcal{E}^{bp}:=\comp \setminus (-\infty,0]$, and we denote the analytic continuation by the same symbol $G_{\bm{\beta}'_{p,q}}$. 
Then \begin{align}
&G_{\bm{\beta}'_{p,q}}(z)
=  \widetilde{G}_{\bm{\beta}'_{p,q}}(z) - \frac{2\pi i }{B(p,q)}\frac{z^{p-1}}{(1+z)^{p+q}}, ~~z \in \comp^{-}. \label{eqbetap} 
\end{align}
\end{enumerate}
All the powers $w \mapsto w^r$ are the principal values in the above statements. 
\end{prop}
\begin{proof} 
(1)\,\,\, Note that the density function $\frac{1}{B(p,q)}w^{p-1}(1-w)^{q-1}$ extends analytically to $\mathcal{D}^b$. Therefore 
for any $z \in\comp^+$, the Cauchy transform $G_{\bm{\beta}_{p,q}}$ can be written as 
$$
 \frac{1}{B(p,q)}\int_{\gamma}\frac{1}{z-w} \,w^{p-1}(1-w)^{q-1}\,dw, 
$$
where $\gamma$ is any simple arc contained in $\comp^-$ except its endpoints $0,1$. This gives the analytic continuation of $G_{\bm{\beta}_{p,q}}$ to the domain containing $\comp^+$, surrounded by $\gamma$ and $(-\infty,0] \cup [1,\infty)$. Since $\gamma$ is arbitrary, we obtain the analytic continuation to the domain $\mathcal{D}^b$. 

For any $z \in \comp^-$, take a simple arc $\gamma$ contained in $\comp^-$ with endpoints 0,1 such that the simple closed curve $\widetilde{\gamma}:=\gamma \cup [0,1]$ surrounds $z$. Then from the residue theorem, we have 
$$
 \frac{1}{B(p,q)}\int_{\widetilde{\gamma}}\frac{1}{z-w}w^{p-1}(1-w)^{q-1}\,dw= \frac{2\pi i}{B(p,q)}z^{p-1}(1-z)^{q-1}, 
$$
showing (\ref{eqbeta}) since the left hand side is equal to $\widetilde{G}_{\bm{\beta}_{p,q}}(z)-G_{\bm{\beta}_{p,q}}(z)$. 
The proof of (2) is similar. 
\end{proof}

Differential equations for Cauchy transforms are crucial to show (\ref{diffA}) and (\ref{diffC}). 
\begin{lem}\label{dif}
 The Cauchy transforms $\widetilde{G}_{\bm{\beta}_{p,q}}, \widetilde{G}_{\bm{\beta}'_{p,q}}$ satisfy the following differential equations: 
\begin{align}
 &\frac{d}{dz}\widetilde{G}_{\bm{\beta}_{p,q}}(z)=\left( \frac{p-1}{z}+\frac{q-1}{z-1}\right) \widetilde{G}_{\bm{\beta}_{p,q}}(z)-\frac{p+q-1}{z(z-1)}, \quad z\in \comp^+, \label{beta}\\
&\frac{d}{dz}\widetilde{G}_{\bm{\beta}'_{p,q}}(z)= \left( \frac{p-1}{z}-\frac{p+q}{z+1}\right) \widetilde{G}_{\bm{\beta}'_{p,q}}(z)+\frac{q}{z(z+1)} \label{betap} \\ 
&~~~~~~~~~~~~\,~=\frac{q(q+1)}{(p+q)z}\left( \left( -z+\frac{p-1}{q+1} \right) \widetilde{G}_{\bm{\beta}'_{p,q+1}}(z)+1 \right), \quad z\in\comp^+. \label{betap2}
 \end{align}
\end{lem}
\begin{proof}
Suppose first $p,q>1$. Then, by integration by parts,  
\[
\begin{split}
\frac{d}{dz}\widetilde{G}_{\bm{\beta}_{p,q}}(z)
 &= \frac{1}{B(p,q)}\int_{0}^1  \frac{-1}{(z-x)^2}x^{p-1}(1-x)^{q-1}dx \\
&=  \frac{1}{B(p,q)}\int_{0}^1  \frac{1}{z-x}\left( \frac{p-1}{x}  -\frac{q-1}{1-x}\right)x^{p-1}(1-x)^{q-1}dx. 
\end{split}
\]
By using the identities $\frac{1}{(z-x)x}  =\frac{1}{z}(\frac{1}{z-x} +\frac{1}{x})$ and $\frac{1}{(z-x)(1-x)}  =\frac{1}{1-z}(\frac{1}{z-x} -\frac{1}{1-x})$, we have 
\[
\begin{split}
\frac{d}{dz}\widetilde{G}_{\bm{\beta}_{p,q}}(z)
&=  \frac{1}{B(p,q)} \frac{p-1}{z}\int_{0}^1  \left( \frac{1}{z-x}  +\frac{1}{x}\right)x^{p-1}(1-x)^{q-1}dx \\
&~~~~~+\frac{1}{B(p,q)} \frac{q-1}{z-1}\int_{0}^1 \left( \frac{1}{z-x}  -\frac{1}{1-x}\right)x^{p-1}(1-x)^{q-1}dx \\
&=\left( \frac{p-1}{z}+\frac{q-1}{z-1}\right) \widetilde{G}_{\bm{\beta}_{p,q}}(z)+\frac{(p-1)B(p-1,q)}{B(p,q)}\frac{1}{z}-\frac{(q-1)B(p,q-1)}{B(p,q)}\frac{1}{z-1}\\
&=\left( \frac{p-1}{z}+\frac{q-1}{z-1}\right) \widetilde{G}_{\bm{\beta}_{p,q}}(z)-\frac{p+q-1}{z(z-1)}. 
\end{split}
\]
Since $\widetilde{G}_{\bm{\beta}_{p,q}}$ and its derivative depend analytically on $p,q>0$, the above differential equation holds for any $p,q>0$.

A similar argument is possible for $\bm{\beta}_{p,q}'$. Suppose first that $p>1$ and then we have   
\[
\begin{split}
\frac{d}{dz}\widetilde{G}_{\bm{\beta}'_{p,q}}(z)
&=  \frac{1}{B(p,q)} \frac{p-1}{z}\int_{0}^\infty  \left( \frac{1}{z-x}  +\frac{1}{x}\right)x^{p-1}(1+x)^{-p-q}dx \\
&~~~~~-\frac{1}{B(p,q)} \frac{p+q}{z+1}\int_{0}^\infty \left( \frac{1}{z-x}  +\frac{1}{1+x}\right)x^{p-1}(1+x)^{-p-q}dx \\
&= \left( \frac{p-1}{z}-\frac{p+q}{z+1}\right) \widetilde{G}_{\bm{\beta}'_{p,q}}(z)+\frac{(p-1)B(p-1,q+1)}{zB(p,q)}-\frac{(p+q)B(p,q+1)}{(z+1)B(p,q)}\\
&= \left( \frac{p-1}{z}-\frac{p+q}{z+1}\right) \widetilde{G}_{\bm{\beta}'_{p,q}}(z)+\frac{q}{z(z+1)}. 
\end{split}
\]
The above equation holds for $p,q>0$ too because of the analytic dependence on $p>0$. 

The second equality (\ref{betap2}) follows from the recursive relation
\begin{equation}\label{rec}
\widetilde{G}_{\bm{\beta}'_{p,q}}(z)
 =  \frac{q}{p+q}\left( (1+z)\widetilde{G}_{\bm{\beta}'_{p,q+1}}(z)-1 \right), 
\end{equation}
which is justified by the following calculation:  
\[
\begin{split}
\widetilde{G}_{\bm{\beta}'_{p,q}}(z)
 &= \frac{1}{B(p,q)}\int_{0}^\infty  \frac{1+z+x-z}{z-x}\frac{x^{p-1}}{(1+x)^{p+q+1}}dx \\ 
 &=  (1+z)\cdot \frac{B(p,q+1)}{B(p,q)}\cdot\frac{1}{B(p,q+1)}\int_{0}^\infty  \frac{1}{z-x}\frac{x^{p-1}}{(1+x)^{p+q+1}}dx \\
 &~~~~- \frac{B(p,q+1)}{B(p,q)}\cdot\frac{1}{B(p,q+1)}\int_{0}^\infty\frac{x^{p-1}}{(1+x)^{p+q+1}}dx \\ 
 &=  \frac{q}{p+q}\left( (1+z)\widetilde{G}_{\bm{\beta}'_{p,q+1}}(z)-1 \right). 
\end{split}
\]
%Note that $\frac{B(p,q+1)}{B(p,q)} = \frac{q}{p+q}$. 
\end{proof}

\begin{lem}\label{lem3}
\begin{enumerate}[\rm(1)]
\item The Cauchy transform of $\bm{\beta}_{p,q}$ satisfies conditions (\ref{diffA}) and (\ref{diffC}) for the domain $\mathcal{D}^b=\mathcal{E}^b$ if $p+q > 2$.\footnote{Condition (\ref{diffC}) holds under the weaker assumption $p+q \neq 1,2$, but we do not need this result.}

 \item The Cauchy transform of $\bm{\beta}'_{p,q}$ satisfies conditions (\ref{diffA}) and (\ref{diffC}) for the domain $\mathcal{D}^{bp}=\mathcal{E}^{bp}$ for any $p,q>0$.  
 \end{enumerate}
\end{lem}
\begin{proof} 
(1)\,\,\, By analyticity, the differential equation (\ref{beta}) holds for $G_{\bm{\beta}_{p,q}}$ in $\mathcal{D}^b$. 
Assume that $z \in \mathcal{D}^b$, $G_{\bm{\beta}_{p,q}}'(z)=0$ and $G_{\bm{\beta}_{p,q}}(z)\in\comp^-$ at the same time. 
The differential equation in Lemma \ref{dif} implies 
\begin{equation} \label{eq980}
G_{\bm{\beta}_{p,q}}(z) = \frac{p+q-1}{(p+q-2)z -p+1}.
\end{equation} 
If $z \in \comp^+$, then $ F_{\bm{\beta}_{p,q}}(z) = \frac{p+q-2}{p+q-1}z -\frac{p-1}{p+q-1}$, which contradicts Proposition \ref{prop08}(\ref{p2}). If $z\in \comp^- \cup (0,1)$, then $\displaystyle G_{\bm{\beta}_{p,q}}(z)\in \comp^+ \cup \real \cup \{\infty\}$ from (\ref{eq980}), a contradiction to the assumption. 
This argument verifies condition (\ref{diffA}). 

Condition (\ref{diffC}) is similar. Assume that $z\in\mathcal{E}^b$ and $\frac{d}{dz}G_{\bm{\beta}_{p,q}}(z)=0$. (i) If $z\in(0,1)$, then $G_{\bm{\beta}_{p,q}}(z) \in \comp^-$ from the Stieltjes inversion formula (\ref{eqtau}). This contradicts (\ref{eq980}), so $z\in(0,1)$ never happens.  
(ii) If $z \in \comp^+$, then $G_{\bm{\beta}_{p,q}}(z) \in \comp^-$.  
(iii) If $z\in\comp^-$, then $G_{\bm{\beta}_{p,q}}(z)$ belongs to $\comp^+$ from (\ref{eq980}). Therefore the assumption $\frac{d}{dz}G_{\bm{\beta}_{p,q}}(z)=0, z \in \mathcal{E}^b$ implies that $G_{\bm{\beta}_{p,q}}(z) \notin \real$.

(2)\,\,\, A similar reasoning applies to $G_{\bm{\beta}'_{p,q}}$. Now assume that $z\in\mathcal{D}^{bp}$, $G_{\bm{\beta}'_{p,q}}(z)\in\comp^-$ and $\frac{d}{dz}G_{\bm{\beta}'_{p,q}}(z)=0$. It follows that 
$$
G_{\bm{\beta}'_{p,q}}(z) = \frac{q}{(q+1)z -p+1},~~~G_{\bm{\beta}'_{p,q+1}}(z) = \frac{q+1}{(q+1)z -p+1}.
$$ 
If $z \in \comp^+$, then $F_{\bm{\beta}'_{p,q+1}}(z) = z -\frac{p-1}{q+1}$, a contradiction to Proposition \ref{prop08}(\ref{p2}), since $\bm{\beta}'_{p,q+1}$ is not a delta measure. 
If $z \in \comp^- \cup (0,\infty)$, then $ G_{\bm{\beta}'_{p,q}}(z) = \frac{q}{(q+1)z -p+1} \in \comp^+ \cup \real\cup\{\infty\}$ which again contradicts the assumption. 

Condition (\ref{diffC}) is also similar. 
\end{proof}

Theorem \ref{thm1}(\ref{thm1-1}) follows from the following stronger fact. 
\begin{thm}\label{thm11}
The beta distribution $\bm{\beta}_{p,q}$ is in $\mathcal{UI}$ in the following cases: (i) $p, q \geq \frac{3}{2}$; (ii) $0<p\leq \frac{1}{2},~p+q \geq 2$; (iii) $0<q\leq \frac{1}{2},~p+q \geq 2$. 
\end{thm}
\begin{proof}
 Assume moreover that $p,q \notin \mathbb{Z}$, $p,q \notin\{\frac{1}{2}, \frac{3}{2}, \frac{5}{2}, \cdots\}$ and $p+q>2$, because these assumptions simplify the proof. We can recover these exceptions by using the fact that the class $\mathcal{UI}$ is closed with respect to the weak convergence \cite{AHb}. 

\emph{Step 1}. In order to construct a domain $\mathcal{D}$ satisfying condition (\ref{A}), we show the following as preparation: 
\begin{align} 
& G_{\bm{\beta}_{p,q}}(x - i0) \notin \real,~~ x \in\real\setminus\{0,1\},\label{eq12} \\ 
& \lim_{z\to 0, z \in \mathcal{E}^b}G_{\bm{\beta}_{p,q}}(z)
=\begin{cases}\label{eq13} 
\infty, &0<p < \frac{1}{2},\\
-\frac{p+q-1}{p-1},& p > \frac{3}{2}, 
\end{cases}\\
& \lim_{z\to 1, z \in \mathcal{E}^b}G_{\bm{\beta}_{p,q}}(z)
=\begin{cases}\label{eq14} 
\infty, &0<q < \frac{1}{2},\\
\frac{p+q-1}{q-1},& q > \frac{3}{2}, 
\end{cases} \\
& \lim_{z \to \infty, z\in\comp^{-}}G_{\bm{\beta}_{p,q}}(z)=\infty. \label{eq15}
\end{align}

(\ref{eq12})\,\,\, From the Stieltjes inversion formula (\ref{eqtau}), we have $\text{Im}\, G_{\bm{\beta}_{p,q}}(x-i0) <0$ for $x\in (0,1)$. 
From Proposition \ref{prop78} and the Stieltjes inversion formula, it follows that 
\begin{equation}\label{eq006}
\text{Im}\, G_{\bm{\beta}_{p,q}}(x-i0)
= 
\begin{cases}
\frac{2\pi}{B(p,q)}\cos(\pi p) |x|^{p-1}(1-x)^{q-1},&x<0, \\[5pt]
\frac{2\pi}{B(p,q)}\cos(\pi q) x^{p-1}(x-1)^{q-1},&x>1. 
\end{cases}
\end{equation}
Hence $\text{Im}\, G_{\bm{\beta}_{p,q}}(x-i0)\neq 0$ for $x\in(-\infty, 0)\cup(1,\infty)$ since we have assumed $p,q \notin\{\frac{1}{2}, \frac{3}{2}, \frac{5}{2}, \cdots\}$. 

(\ref{eq13})\,\,\, From (\ref{15.3.7}) and (\ref{15.3.3}) one has
%\footnote{When $p \in \mathbb{Z}$, the following identity requires modification because $F(a,b;c;z)$ is not defined for $c \in\{0,-1,-2,\cdots\}$, but we avoid this case.}
\begin{equation}\label{eq123}
\begin{split}
G_{\bm{\beta}_{p,q}}(z)
&=\frac{1}{z}  F(1,p;p+q; z^{-1}) \\
&= \frac{1}{z}\left(\frac{\Gamma(p+q)\Gamma(p-1)}{\Gamma( p)\Gamma(p+q-1)} \left(-\frac{1}{z}\right)^{-1} F(1,2-p-q;2-p;z) \right. \\
&~~~~~~~~+\left.\frac{\Gamma(p+q)\Gamma(1-p)}{\Gamma(1)\Gamma( q)}  \left(-\frac{1}{z}\right)^{-p} F(p,1-q;p;z) \right)\\
&= -\frac{p+q-1}{p-1}F(1,2-p-q;2-p;z) -\frac{\pi}{B(p,q)\sin\pi p}   (-z)^{p-1}(1-z)^{q-1},  
\end{split}
\end{equation}
which is valid in $\comp^+$. By using the integral representation (\ref{int}), the RHS of (\ref{eq123}) is analytic in $\mathcal{E}^b$ and hence   gives the analytic continuation of $G_{\bm{\beta}_{p,q}}$ as claimed in Proposition \ref{prop78}. 
Eq.\ (\ref{eq13}) easily follows from (\ref{eq123}).

(\ref{eq14})\,\,\, We now use formula (\ref{15.3.9}) to obtain 
\begin{equation}\label{eq1230}
\begin{split}
G_{\bm{\beta}_{p,q}}(z)
&= \frac{p+q-1}{q-1}F(1,2-p-q;2-q;1-z) +\frac{\pi}{B(p,q)\sin\pi q} z^{p-1}(z-1)^{q-1},   
\end{split}
\end{equation}
which is valid in $\comp^+$. Since the RHS is analytic in $\mathcal{E}^b$, it gives  the analytic continuation of $G_{\bm{\beta}_{p,q}}$ as claimed in Proposition \ref{prop78}. Eq.\ (\ref{eq14}) follows from (\ref{eq1230}). 

(\ref{eq15})\,\,\, Finally, (\ref{eq15}) follows from Proposition \ref{prop78}. Note that $\lim_{z\to\infty, z\in\comp^-}\widetilde{G}_{\bm{\beta}_{p,q}}(z)=0$ because $\bm{\beta}_{p,q}$ is compactly supported. 

From (\ref{eq123}), the asymptotics of $G_{\bm{\beta}_{p,q}}$ as $z\to0$ is as follows:  
\begin{equation}\label{eq124}
G_{\bm{\beta}_{p,q}}(z)
= 
\begin{cases}
- \frac{\pi}{B(p,q) \sin\pi p} (-z)^{p-1}+o(|z|^{p-1}), & 0<p<\frac{1}{2},\\[5pt]
-\frac{p+q-1}{p-1} - \frac{\pi}{B(p,q) \sin\pi p} (-z)^{p-1} +o(|z|^{p-1}),&\frac{3}{2}<p<2,  \\[5pt]
-\frac{p+q-1}{p-1} -\frac{(p+q-1)(p+q-2)}{(p-1)(p-2)}z +o(|z|),&p>2.
\end{cases}
\end{equation}

\emph{Step 2.} 
We are going to find a simple curve $C \subset \real \cup \comp^-$ such that $G_{\bm{\beta}_{p,q}}$ maps $C \cup \{\infty\}$ into $\real\cup\{\infty\}$. 
If such a curve exists, then the Jordan domain $\comp^+ \subset D(C{}) \subset \mathcal{E}^b$ surrounded by $C\cup\{\infty\}$ satisfies condition (\ref{A}), so that we can use Proposition \ref{prop91}.

\textbf{Case $p,q >\frac{3}{2}$.} From elementary calculus, $G_{\bm{\beta}_{p,q}}(x+i0)$ is injective on $(-\infty,0] \cup [1,\infty)$, taking real values, and we have $G_{\bm{\beta}_{p,q}}((-\infty,0]+i0)=[-\frac{p+q-1}{p-1},0)$ and $G_{\bm{\beta}_{p,q}}([1,\infty)+i0)=(0,\frac{p+q-1}{q-1}]$. 

First we consider the case $\frac{3}{2}<p<2$. Let $\mathcal{S}_{J}(R{})$ be the sector $\{z\in \comp\setminus\{0\}: \arg z \in J,|z|<R\}$ and $\mathcal{S}_{J}$ be $\{z\in \comp\setminus\{0\}: \arg z \in J\}$ for $J \subset (-\pi, \pi)$.\footnote{This definition will be generalized in Section \ref{subsec12}.} 
The function $-\frac{\pi}{B(p,q) \sin\pi p} (-z)^{p-1}$ maps the sector $\mathcal{S}_{\left(-\frac{2-p}{p-1}\pi, \pi\right)}$ bijectively onto $\comp^-$, mapping the half line $e^{-i\frac{2-p}{p-1}\pi}(0,\infty)$ onto $(-\infty,0).$ 
Even under the perturbation $o(|z|^{p-1})$, for any small $\eta >0$ satisfying $-\pi<-\frac{2-p}{p-1}\pi -\eta<-\frac{2-p}{p-1}\pi +\eta <0$, we can find $\varepsilon, \delta>0$ such that $G_{\bm{\beta}_{p,q}}$ maps the boundary of  $\mathcal{S}_{\left(-\frac{2-p}{p-1}\pi -\eta, -\frac{2-p}{p-1}\pi +\eta\right)}(\varepsilon)$ to a curve surrounding each point of $(-\frac{p+q-1}{p-1}-\delta,-\frac{p+q-1}{p-1})$ just once. Hence, we can find a curve $c_1^\delta \subset \mathcal{S}_{\left(-\frac{2-p}{p-1}\pi -\eta, -\frac{2-p}{p-1}\pi +\eta\right)}(\varepsilon)$ such that $G_{\bm{\beta}_{p,q}}(c_1^\delta)=(-\frac{p+q-1}{p-1}-\delta,-\frac{p+q-1}{p-1})$. 
By letting $\varepsilon,\delta$ smaller, we know that an endpoint of $c_1^\delta$ is $0$. 
%This curve is tangent to the half line $\{re^{-i\frac{2-p}{p-1}\pi }: r\geq 0\}$ at 0. 
For $p>2$, we can find a curve $c_1^\delta$ similarly.

As in Proposition \ref{prop91}, we can prolong the curve $c_1^\delta$ by using property (\ref{diffC}), to obtain a maximal curve $c_1$ that is mapped into $(-\infty,-\frac{p+q-1}{p-1})$ injectively by $G_{\bm{\beta}_{p,q}}$. Suppose $G_{\bm{\beta}_{p,q}}(c_1)=(x_0,-\frac{p+q-1}{p-1})$ for some $x_0\in (-\infty,-\frac{p+q-1}{p-1})$. For each point $x\in (x_0,-\frac{p+q-1}{p-1})$, let $u \in c_1$ denote the preimage of $x$. 
 The following cases are possible: 
\begin{enumerate}[\rm(i)] 
\item\label{casei} When $x$ converges to $x_0$, the preimages $u$ have an accumulative point $u_0$ in $\mathcal{E}^b$; 

\item\label{caseii}  When $x$ converges to $x_0$, the preimages $u$ have an accumulative point $u_1$ in $\partial \mathcal{E}^b \cup \{\infty \}$.  
\end{enumerate}
In the case (\ref{casei}), we can extend the curve $c_1$ more because of condition (\ref{diffC}) and the obvious fact $G_{\bm{\beta}_{p,q}}(u_0)=x_0$; a contradiction to the maximality of $c_1$.

In the case (\ref{caseii}), property (\ref{eq12}) implies that $c_1$ cannot approach $\real\setminus\{0,1\}$ and so $u_1\notin (-\infty,0) \cup (1,\infty)$. Moreover, from (\ref{eq13}) and (\ref{eq14}), 
 $u_1$ never be $0$ or $1$ because $G_{\bm{\beta}_{p,q}}$ is injective in $c_1$.  Thus we conclude that $u_1=\infty$ and so $x_0 = -\infty$, a contradiction to the assumption $x_0 >-\infty$. Consequently, $G_{\bm{\beta}_{p,q}}(c_1)$ coincides with $(-\infty,-\frac{p+q-1}{p-1})$ and $c_1$ connects $0$ and $\infty$.\footnote{Here we also need the fact that $G_{\bm{\beta}_{p,q}}$ is analytic in $\mathcal{E}^b$, not only meromorphic.}

Similarly, starting from $1$, we get a curve $c_2$ connecting $1$ and $\infty$ such that $G_{\bm{\beta}_{p,q}}$ maps $c_2$ bijectively onto $(\frac{p+q-1}{q-1}, \infty)$. Therefore, $G_{\bm{\beta}_{p,q}}$ maps  $C:=((-\infty,0]+i0) \cup c_1 \cup c_2 \cup ([1,\infty)+i0)\cup\{\infty\}$ bijectively onto $\real\cup \{\infty\}$ and so $\bm{\beta}_{p,q} \in \mathcal{UI}$. 

\begin{figure}[htpb]
\begin{minipage}{0.5\hsize}
\begin{center}
\includegraphics[width=0.82\linewidth,clip]{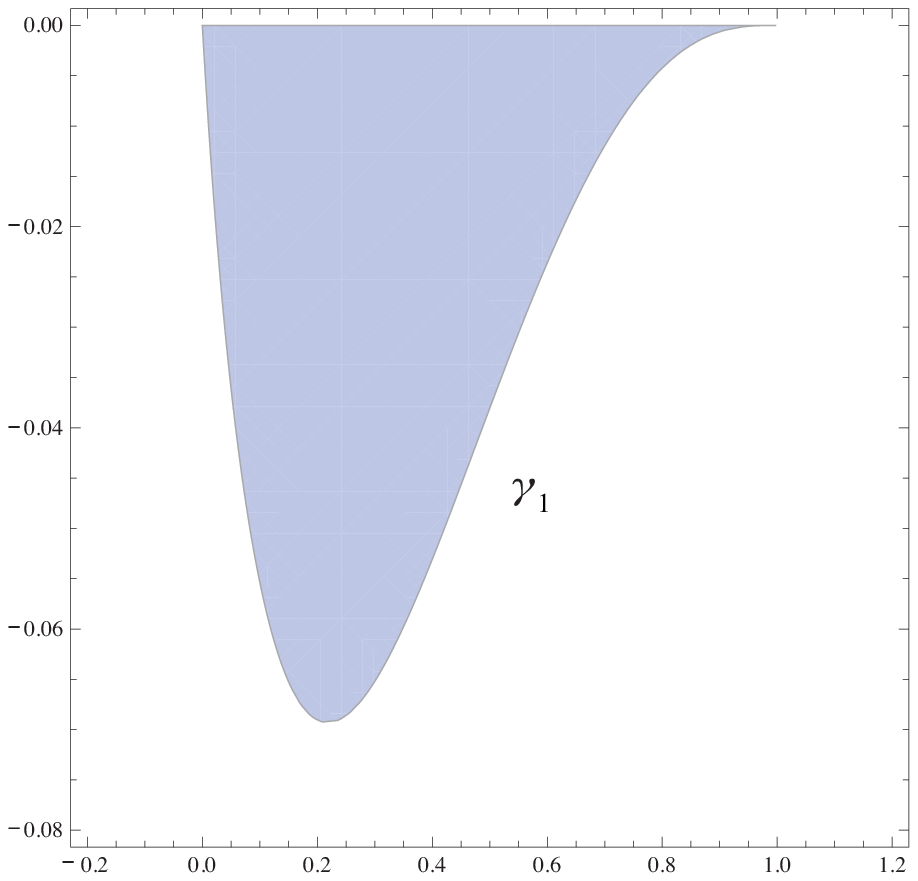}
\caption{A curve for $\bm{\beta}_{0.2,4}$. Figures are drawn by Mathematica 8.0 and Illustrator CS6.}\label{b1}
\end{center}
\end{minipage}
\begin{minipage}{0.5\hsize}
\begin{center}
\includegraphics[width=0.86\linewidth,clip]{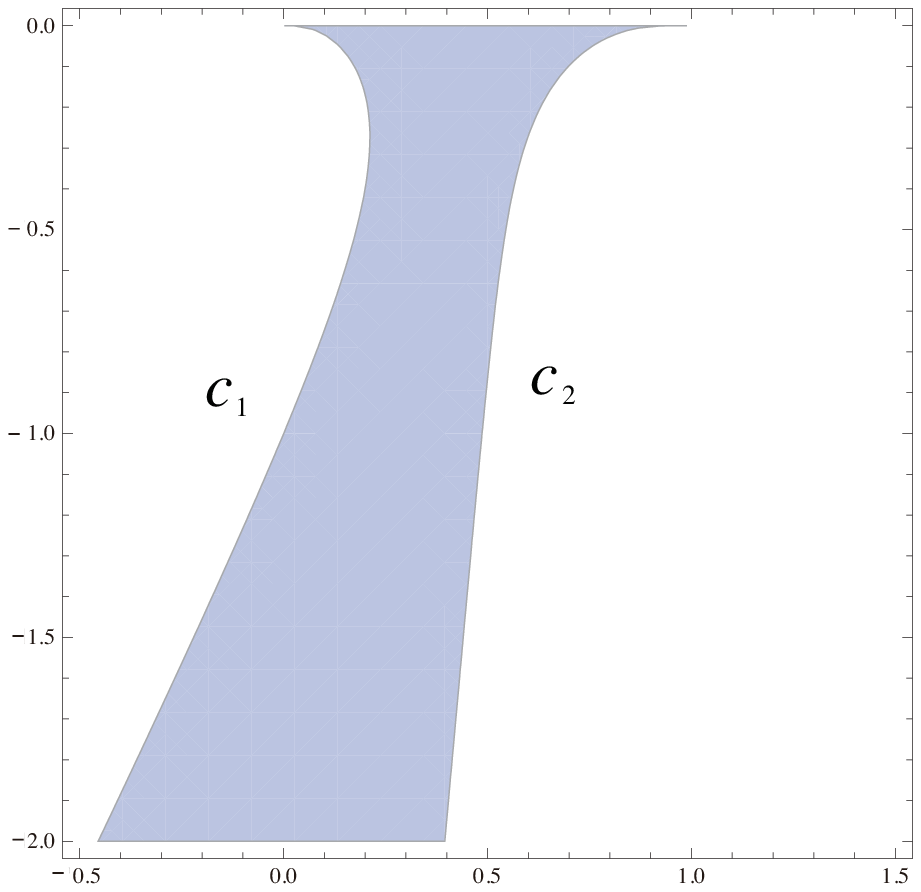}
\caption{Curves for $\bm{\beta}_{4.5,6}$}\label{b2}
\end{center}
\end{minipage}
\end{figure}

\textbf{Case $0<p<\frac{1}{2},~p+q>2$.}  From elementary calculus, the map $x\mapsto G_{\bm{\beta}_{p,q}}(x+i0)$ is injective on $(-\infty,0] \cup [1,\infty)$, taking real values, and we have $G_{\bm{\beta}_{p,q}}((-\infty,0]+i0)=(-\infty, 0]$, $G_{\bm{\beta}_{p,q}}([1,\infty)+i0)=(0,\frac{p+q-1}{q-1}]$. 

The function $-\frac{\pi}{B(p,q) \sin\pi p} (-z)^{p-1}$ now maps the sector $\mathcal{S}_{\left(-\frac{p}{1-p}\pi, \pi\right)}$ bijectively onto $\comp^-$, mapping the half line $e^{-\frac{p\pi i}{1-p}}(0,\infty)$ onto $(0,\infty).$ 
Similarly to $c_1^\delta$, for any $\eta>0$ such that $-\pi < -\frac{p}{1-p}\pi -\eta< -\frac{p}{1-p}\pi +\eta <0$, there exist $\varepsilon, R>0$ and a curve $\gamma_1^R \subset \mathcal{S}_{\left(-\frac{p}{1-p}\pi -\eta, -\frac{p}{1-p}\pi +\eta\right)}(\varepsilon)$, starting from $0$, such that $G_{\bm{\beta}_{p,q}}(\gamma_1^R)=(R,\infty)$. We can prolong the curve $\gamma_1^R$ by using property (\ref{diffC}), to obtain a maximal curve $\gamma_1$ which is mapped by $G_{\bm{\beta}_{p,q}}$ into $(-\infty,\infty)$ injectively, and it has an endpoint $0$. The other endpoint, denoted by $v_1$, is $0$, $1$ or $\infty$. 

The case $v_1=0$ means $\gamma_1$ starts from 0 and goes back to $0$ again. This never happens because the dominating term $-\frac{\pi}{B(p,q) \sin\pi p} (-z)^{p-1}$ defined in $\mathcal{E}^b$ takes real values only on the half line $e^{-\frac{p\pi i}{1-p}}(0,\infty)$. 

If $v_1=1$, then the curve $C :=((-\infty,0]+i0) \cup \gamma_1 \cup( [1,\infty)+i0)$ is simple and $G_{\bm{\beta}_{p,q}}$ bijectively maps $C\cup\{\infty\}$ onto $\real\cup\{\infty\}$ as in Fig.\ \ref{b1}.

If $v_1=\infty$,\footnote{Numerical computation indicates that this case never happens (see Fig.\ \ref{b1}).}  
then $G_{\bm{\beta}_{p,q}}(\gamma_1) = (-\infty,\infty)$ from (\ref{eq15}). In this case, we will construct another curve $\gamma_2$ starting from $1$ such that $G_{\bm{\beta}_{p,q}}(\gamma_2) \subset (\frac{p+q-1}{q-1},\infty)$. 
If the other endpoint of $\gamma_2$ is $0$, then the curve $((-\infty,0]+i0) \cup \gamma_2 \cup ([1,\infty)+i0)$ is enough for our purpose. If the other endpoint is $\infty$, then $G_{\bm{\beta}_{p,q}}(\gamma_2)=(\frac{p+q-1}{q-1},\infty)$ and $G_{\bm{\beta}_{p,q}}$ is not injective on $\gamma_1 \cup \gamma_2$. Let $\gamma:=\gamma_1 \cup \gamma_2$ if $\gamma_1$ and $\gamma_2$ are disjoint. If $\gamma_1$ and $\gamma_2$ cross, then we define a simple curve $\gamma$ starting from $0$, going along $\gamma_1$ until the first crossing point, and then along $\gamma_2$, to arrive at $1$. Now the Jordan domain surrounded by $\Gamma:=((-\infty,0]+i0) \cup \gamma \cup( [1,\infty)+i0)$ satisfies condition (\ref{A}), and hence $\bm{\beta}_{p,q} \in \mathcal{UI}$ from Proposition \ref{prop91}.  

The case $0<q<\frac{1}{2}, p+q>2$ is similar by symmetry. 
\end{proof}

\begin{rem}
Let $\Gamma$ be a Jordan closed curve in $\comp$ and $f$ be analytic in $d(\Gamma)$, the bounded Jordan domain surrounded by $\Gamma$, and suppose $f$ is continuous on $\overline{d(\Gamma)}$. If $f$ is injective on $\Gamma$, then $f$ maps $\overline{d(\Gamma)}$ bijectively onto the image $f(\overline{d(\Gamma)})$ and $f(d(\Gamma))$ is the domain surrounded by the Jordan curve $f(\Gamma)$; see \cite{B79}, p.\ 309, 310. 
However, this result does not hold for unbounded domains. Take $f(z):=z^2$ for instance. Now $f$ maps $i(-\infty,0]\cup[0,\infty) \cup\{\infty\}$ bijectively onto $\real \cup \{\infty\}$, but $f$ does not map $\{z\in\comp: \arg z \in(0,\frac{3\pi}{2})\}$ onto $\comp^+$ bijectively. 
In case $p,q >\frac{3}{2}$, we have constructed simple curves $c_1, c_2$, starting from $0$ and $1$ respectively, such that $G_{\bm{\beta}_{p,q}}$ maps the curve(s) $C=(-\infty,0] \cup c_1 \cup c_2 \cup [1,\infty)\cup\{\infty\}$ bijectively onto $\real\cup \{\infty\}$. This however is not sufficient to conclude that $G_{\bm{\beta}_{p,q}}$ maps $D(C{})$ bijectively onto $\comp^+$. So we need to use Proposition \ref{prop91}. 
\end{rem}
\begin{rem}\label{rem09}
The computation (\ref{eq006}) shows that the domain $\mathcal{D}^b$ satisfies condition (\ref{barrierA}) if $(p,q)\in \mathcal{R}:=\left(\left(0,\frac{1}{2}\right) \cup \bigcup_{n=1}^\infty \left(2n- \frac{1}{2}, 2n+\frac{1}{2}\right)\right)^2$. Hence, if $(p,q)$ is in the closure of $\mathcal{R}\cap\{(x,y): x+y>2\}$, the proof is much shorter because one does not need Step 2. However, if $(p,q) \notin\mathcal{R}$, the domain $\mathcal{D}^b$ does not satisfy (\ref{barrierA}). In this case, we need condition (\ref{C}) to find an alternative domain $\mathcal{D}$ for condition (\ref{A}). Such a domain $\mathcal{D}$ was realized as $D({C})$. 
\end{rem}

A similar proof applies for beta prime distributions too. 

\begin{thm}\label{thm12} 
The beta prime distribution $\bm{\beta}'_{p,q}$ belongs to class $\mathcal{UI}$ if  $p\in (0,\frac{1}{2}] \cup [\frac{3}{2}, \infty)$. 
\end{thm}
\begin{proof}
Assume moreover that $p\notin\mathbb{Z}$, $p,q \notin\{\frac{1}{2}, \frac{3}{2}, \frac{5}{2}, \cdots\}$ and
take $\mathcal{E}^{bp}=\comp\setminus (-\infty,0]$. As in the proof for beta distributions, we will construct a good domain $\mathcal{D}$ to apply Proposition \ref{prop91}(1). 

\emph{Step 1}. We are going to show that 
\begin{align}
& G_{\bm{\beta}'_{p,q}}(x - i0)\in\comp\setminus \real,~~ x \in \real\setminus\{-1,0\}, \label{eq21}\\
& \lim_{z\to-1,z\in\comp^-}G_{\bm{\beta}'_{p,q}}(z)=\infty, \label{eq210} \\
& \lim_{z\to 0, z \in\mathcal{E}^{bp}}G_{\bm{\beta}'_{p,q}}(z)
=\begin{cases}\label{eq22} 
\infty, &0<p < \frac{1}{2},\\
-\frac{q}{p-1},& p > \frac{3}{2}, 
\end{cases}\\
&G_{\bm{\beta}'_{p,q}}(z) = \frac{1}{z}(1+o(1)), \label{eq23}
\end{align}
where $o(1)$ means that, for any $\varepsilon>0$, there exists $R>0$ such that $|o(1)| \leq \varepsilon$ for $z\in (\comp\setminus (-\infty,0]) \cap\{z: |z|>R\}$. 

(\ref{eq21})\,\,\, The first identity in Proposition \ref{prop113}(2)  leads to 
\begin{equation}\label{eq987}
\text{Im}\,G_{\bm{\beta}'_{p,q}}(x-i0)
=\frac{1}{(1+x)^2}\text{Im}\, G_{\bm{\beta}_{p,q}}\left(\frac{x}{x+1}\right), 
\end{equation}
showing the claim because of the previous computation (\ref{eq12}). 

(\ref{eq210}) follows from Proposition \ref{prop78}(\ref{betapp}).  

(\ref{eq22}) follows from (\ref{eq13}) and Proposition \ref{prop113}(2) with elementary calculus. 
  
(\ref{eq23})\,\,\, The contour $(0, \infty)$ of the integral $\widetilde{G}_{\bm{\beta}'_{p,q}}(z)$ can be displaced to $e^{i\theta_0}(0,\infty)$ for any $\theta_0 \in (0,\pi)$: 
$$
\widetilde{G}_{\bm{\beta}'_{p,q}}(z)  = \int_{e^{i\theta_0}(0,\infty)} \frac{1}{z-x}\,\bm{\beta}'_{p,q}(dx),~~z\in\comp^-. 
$$  
Hence $\widetilde{G}_{\bm{\beta}'_{p,q}}(z) =\frac{1}{z}(1+o(1))$ uniformly as $z \to \infty, ~\arg z \in (-\pi,\frac{1}{2}\theta_0)$, and so $G_{\bm{\beta}'_{p,q}}(z) =\frac{1}{z}(1+o(1))$ as $z\to \infty,~\arg z \in (-\pi,\frac{1}{2}\theta_0)$  
from Proposition \ref{prop78}(\ref{betapp}). The estimate in $\comp^+$ is similar. 

\emph{Step 2}. 
From (\ref{eq124}) and Proposition \ref{prop113}(2), we can write 
\begin{equation}\label{eq25}
G_{\bm{\beta}'_{p,q}}(z) 
= 
\begin{cases}
- \frac{\pi}{B(p,q) \sin\pi p} (-z)^{p-1}+o(|z|^{p-1}), & 0<p<\frac{1}{2},\\[5pt]
-\frac{q}{p-1} - \frac{\pi}{B(p,q) \sin\pi p} (-z)^{p-1} +o(|z|^{p-1}),&\frac{3}{2}<p<2,  \\[5pt]
-\frac{q}{p-1} - \frac{q(q+1)}{(p-1)(p-2)}z +o(|z|), &p>2.
\end{cases}
\end{equation}

\textbf{Case $p>\frac{3}{2}$.} 
For small $\varepsilon>0$, we can find a curve $c_1^{\varepsilon}$ starting from $0$ as in the proof of Theorem \ref{thm11}, such that $G_{\bm{\beta}'_{p,q}}$ maps $c_1^\varepsilon$ into  $(-\frac{q}{p-1}-\varepsilon, -\frac{q}{p-1})$ injectively. The curve $c_1^\varepsilon$ extends to a maximal curve $c_1$ such that $G_{\bm{\beta}'_{p,q}}(c_1) \subset (-\infty, -\frac{q}{p-1})$. An endpoint of $c_1$ is $0$. The other endpoint cannot be in $(-\infty,0)$ because of (\ref{eq21}),  nor be in $\mathcal{E}^{bp}$ because $c_1$ is maximal and $G_{\bm{\beta}'_{p,q}}$ is analytic in $\mathcal{E}^{bp}$, 
and so the other endpoint is $0$, $-1$ or $\infty$. Since $G_{\bm{\beta}'_{p,q}}$ is analytic in $\mathcal{E}^{bp}$ and injective on $c_1$, we conclude that the other endpoint must be $-1$. Hence $G_{\bm{\beta}'_{p,q}}(c_1) = (-\infty, -\frac{q}{p-1})$ from (\ref{eq210}). 

The point of the next step is to find a curve starting from $\infty$. 
From (\ref{eq23}), there exists $\delta>0$ such that $|G_{\bm{\beta}'_{p,q}}(z)-\frac{1}{z}| \leq \frac{1}{4|z|}$ for $z\in\mathcal{S}_{\left(-\frac{\pi}{4}, \frac{\pi}{4}\right)} \cap \{z:|z|>\frac{1}{2\delta}\}$. Hence $G_{\bm{\beta}'_{p,q}}$ maps the boundary of $\mathcal{S}_{\left(-\frac{\pi}{4}, \frac{\pi}{4}\right)} \cap \{z:|z|>\frac{1}{2\delta}\}$ to a curve surrounding each point of $(0,\delta)$ just once, and so there is a curve $c_2^{\delta}\subset \mathcal{S}_{\left(-\frac{\pi}{4}, \frac{\pi}{4}\right)} \cap \{z:|z|>\frac{1}{2\delta}\}$ such that $G_{\bm{\beta}'_{p,q}}(c_2^\delta) = (0,\delta)$. Following the argument of Theorem \ref{thm11}, we can extend $c_2^\delta$ to a curve $c_2$ so that $G_{\bm{\beta}'_{p,q}}$ maps $c_2$ into  $(0,\infty)$ injectively, thanks to property (\ref{C}). The endpoint of $c_2$ is in $(-\infty,0] \cup \{\infty\}$, but it cannot be in $(-\infty,-1) \cup (-1,0)$ because of (\ref{eq21}), nor be $0$ or $\infty$ since $G_{\bm{\beta}'_{p,q}}$ is analytic and injective in $c_2$. Therefore, $c_2$ connects $\infty$ and $-1$, and hence $G_{\bm{\beta}'_{p,q}}(c_2)=(0,\infty)$. 

From (\ref{eq25}), it follows that $G_{\bm{\beta}'_{p,q}}((-\infty,0)+i0)= (-\frac{q}{p-1},0)$. 
Now we know that $G_{\bm{\beta}'_{p,q}}$ maps $C:=((-\infty,0]+i0)\cup c_1\cup c_2\cup \{-1,\infty\}$ bijectively onto $\real\cup \{\infty\}$ and it is analytic in the Jordan domain surrounded by $C$. 
We can now use Proposition \ref{prop91}. 
\begin{figure}[htpb]
\begin{minipage}{0.5\hsize}
\begin{center}
\includegraphics[width=70mm,clip]{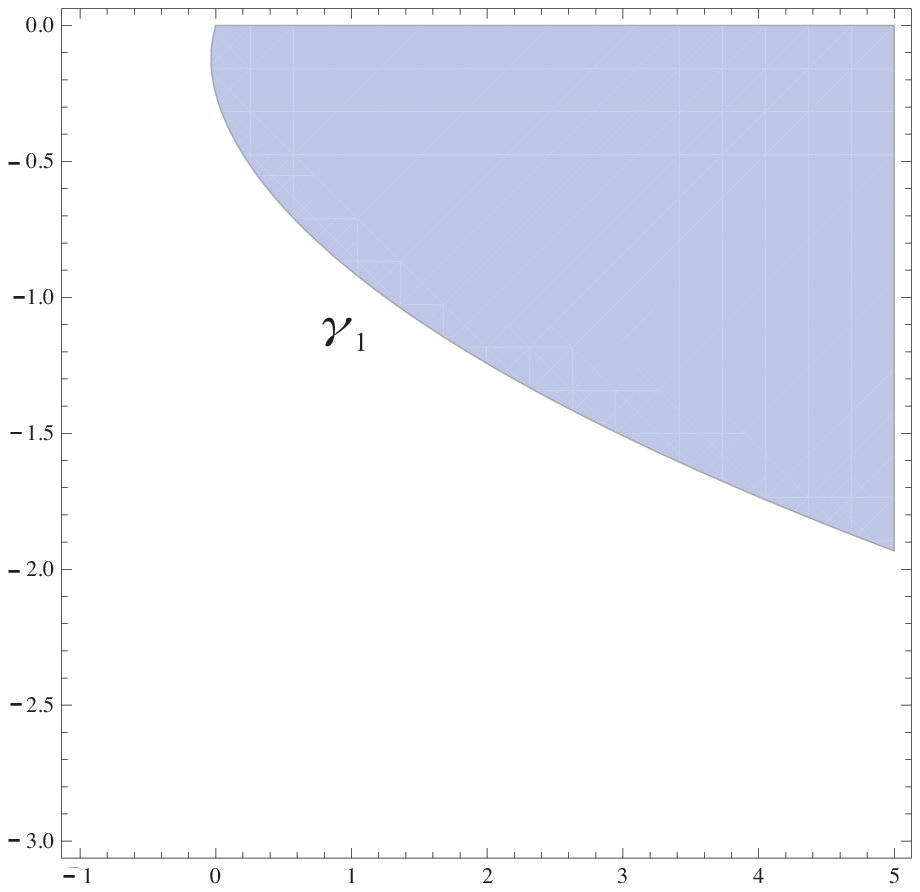}
\caption{A curve for $\bm{\beta}'_{0.4,0.5}$}\label{bp1}
\end{center}
  \end{minipage}
\begin{minipage}{0.5\hsize}
\begin{center}
\includegraphics[width=66.8mm,clip]{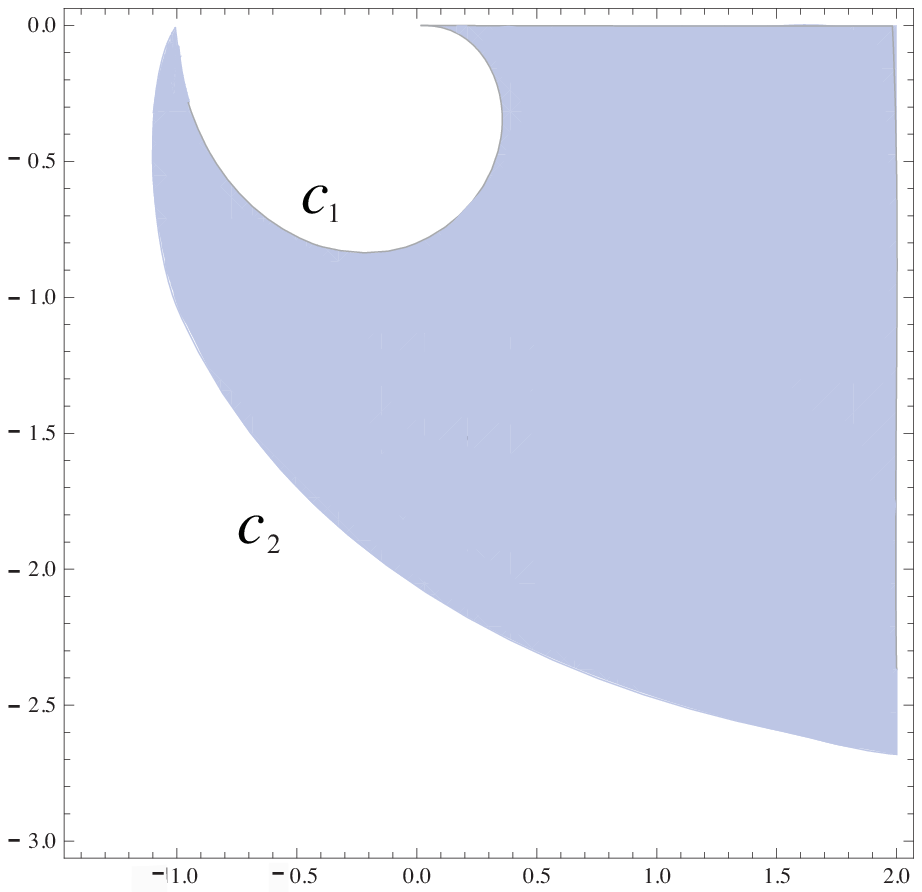}
\caption{Curves for $\bm{\beta}'_{5,2}$}\label{bp2}
\end{center}
\end{minipage}
\end{figure}

\textbf{Case $0<p<\frac{1}{2}$.}  For large $R>0$ and small $\delta>0$, we can construct curves $\gamma_1^R$ and $\gamma_2^\delta$ respectively starting from $0$ and $\infty$, such that $G_{\bm{\beta}'_{p,q}}$ bijectively maps $\gamma_1^R$ onto $(R, \infty)$ and $\gamma_2^\delta$ onto $(0,\delta)$; the construction of $\gamma_1^R$ is similar to the case $0<p<\frac{1}{2},~p+q >2$ in Theorem \ref{thm11}, and the  construction of $\gamma_2^\delta$ is the same as that of $c_2^\delta$ in the case $p>\frac{3}{2}$.  The extension $\gamma_1$ of $\gamma_1^R$ starts at $0$ and the other endpoint is $-1$, $0$ or $\infty$, but $0$ is impossible from the same reasoning as in Theorem \ref{thm11}, 
 and so $G_{\bm{\beta}'_{p,q}}(\gamma_1)=(-\infty,\infty)$ or $G_{\bm{\beta}'_{p,q}}(\gamma_1)=(0,\infty)$. The extension $\gamma_2$ of $\gamma_2^\delta$ starts from $\infty$ and the other endpoint is $-1$ or $0$. If $\gamma_1$ connects $0, \infty$ or if $\gamma_2$ connects $\infty, 0$, one can take $C:=((-\infty,0]+i0)\cup \gamma_i \cup\{0,\infty\}$
 by choosing the corresponding $\gamma_i$, to apply Proposition \ref{prop91}(1) for the domain $D(C{})$. Otherwise, one can still apply Proposition \ref{prop91}(1) similarly to the last part of the proof of Theorem \ref{thm11}.\footnote{Numerical computation indicates that $\gamma_1$ always connects $0$ and $\infty$, and $\gamma_1=\gamma_2$ as shown in Fig.\ \ref{bp1}.} 
\end{proof}

\section{Criteria for non free infinite divisibility} \label{subsec12}
We have seen that beta and beta prime distributions are FID if the parameters $(p,q)$ belong to specific regions. 
By contrast, many distributions outside those regions are not FID as shown in this section. 

\subsection{Method based on a local property of probability density function}
Given a FID measure $\mu$, the following properties are known thanks to Belinschi and Bercovici: the absolutely continuous part of $\mu$ is real analytic wherever it is strictly positive \cite[Theorem 3.4]{BB04}; 
$\mu$ has no singular continuous part \cite[Theorem 3.4]{BB04}; $\mu$ has at most one atom \cite[Theorem 3.1]{BB04}. 

Now, we will deeply study the real analyticity of the density function of a FID measure.   
Basic concepts and notations are defined below. Let $\mathcal{S}^{(n)}$ $(n\in \mathbb{Z})$ denote the open set $\comp\setminus [0,\infty)$ whose element $z$ is endowed with the argument $\arg z \in (2n\pi, 2n\pi +2\pi)$. By identifying the slit $\lim_{y\searrow 0}([0,\infty) +iy)$ of $\mathcal{S}^{(n)}$ and the slit $\lim_{y\searrow 0}([0,\infty) -iy)$ of $\mathcal{S}^{(n-1)}$ for each $n$, we define 
a helix-like Riemannian surface $\mathcal{S}$. We express an element $z\in \mathcal{S}$ uniquely by $z=|z|e^{i\theta}=re^{i\theta}$, $|z|=r>0, \theta \in \real$.  The functions 
$z^\alpha=r^\alpha e^{i\alpha \theta}$ $(\alpha \in \comp)$ and $\log z=\log r + i\theta$ can be regarded as analytic maps in $\mathcal{S}$. 
Let $\mathcal{S}_{J}(R{})$ denote the subset $\{z\in \mathcal{S}: \arg z \in J, ~0<|z|<R\}$ for $J \subset \real$, and 
also  $\mathcal{S}_{J}:= \cup_{R>0}\mathcal{S}_J(R{}).$ We understand that $\comp^+ = \mathcal{S}_{(0,\pi)}$ and $(0,\infty)$ is the half line corresponding to $\arg z =0$.

We are ready to state the main theorem of this section, which contributes to Theorem \ref{thm1}.  
\begin{thm}\label{thm90} 
Let $\mu$ be a probability measure on $\real$ whose restriction to an interval $(x_0-\delta, x_0+\delta)$ has a local density function $p(x)$  
 of the form 
\begin{equation}\label{eq910}
p(x)=
\begin{cases}
c(x-x_0)^{\alpha-1}(1+f(x)), & x_0<x< x_0 + \delta, \\
0, &  x_0 - \delta < x <x_0, 
\end{cases}
\end{equation}
where $c,\delta>0$ and $x_0\in\real$. Let $\theta(\alpha):= \left(\frac{1}{|\alpha-1|}-1\right)\pi$, and assume the following:   
\begin{enumerate}[\rm(i)]
\item \label{z1} $\alpha \in \mathcal{I}=\left(\bigcup_{n=1}^\infty(\frac{2n-1}{2n}, \frac{2n}{2n+1}) \right) \cup \left( \bigcup_{n=1}^\infty(\frac{2n+2}{2n+1}, \frac{2n+1}{2n})\right) \subset (\frac{1}{2},\frac{3}{2})$;  
\item\label{z2} $f(\cdot+x_0)$ is real analytic in $(0,\delta)$, and it extends to an analytic map in $\mathcal{S}_{(-\theta(\alpha)-\theta_0,\theta_0)}(\delta)$ for some $\theta_0\in(0,\frac{\pi}{4})$; 
 %\item\label{z3} $f$ takes real values in $(x_0,x_0+\delta)$; 
 \item\label{z4} there is a continuous function $g :[0,\delta) \to [0,\infty)$ such that $g(0)=0$ and $|f(x_0+re^{i\theta})|\leq g(r{})$ for $r\in(0,\delta)$, $\theta \in (-\theta(\alpha)-\theta_0,\theta_0)$.  
 \end{enumerate}
 Then $\mu$ is not FID.  
 \end{thm}
\begin{rem}
 A typical function $f$ satisfying the assumptions (\ref{z2}), (\ref{z4}) is  
$$
(x-x_0)^\beta(-\log(x-x_0))^{\gamma} ~~(\beta>0, \gamma\in \real, x_0<x<x_0+1)
$$
 and absolutely convergent series of such functions. More restrictively,  any real analytic function in a neighborhood of $x_0$, vanishing at $x_0$, satisfies those assumptions. 
\end{rem}
\begin{cor}
The beta distribution $\bm{\beta}_{p,q}$ is not  FID if $p\in  \mathcal{I}$ or $q \in \mathcal{I}$. The beta prime distribution $\bm{\beta}'_{p,q}$ and the gamma distribution $\bm{\gamma}_p$ are not  FID if $p\in \mathcal{I}$. 
\end{cor}
\begin{proof}[Proof of Theorem \ref{thm90}.]
For simplicity we assume that $x_0=0$. Divide $\mu$ into three finite measures as $\mu=\rho+\rho'+\rho''$, where 
$$
\rho:= c x^{\alpha-1} 1_{(0,\delta)}(x)\,dx,~~ \rho':= c x^{\alpha-1}f(x) 1_{(0,\delta)}(x)\,dx, ~~\rho'' :=\mu|_{\real \setminus (-\delta,\delta)}.
$$ 

\emph{Step 1: Analytic continuation of $G_\mu$.} %Assume that $\alpha \in(0,1)\cup (1,2)$ for the moment.  
Since the function $cx^{\alpha-1}$ is proportional to the density function of the beta distribution $\bm{\beta}_{\alpha,1}$, the Cauchy transform $\widetilde{G}_\rho$  is equal to
%\footnote{$\rho$ may be a signed measure, but the Cauchy transform can be defined for a finite Borel signed measure on $\real$.}$
$
\frac{c\delta^\alpha}{\alpha}\frac{1}{z}F(1,\alpha,\alpha+1, \frac{\delta}{z}), 
$
and so 
\begin{equation}\label{eq080}
G_\rho(z) = 
 \frac{c\delta^{\alpha-1}}{1-\alpha} F\left(1,1-\alpha,2-\alpha,\frac{z}{\delta}\right) -  \frac{c\pi }{\sin \pi\alpha}(-z)^{\alpha-1},
\end{equation}
where the formulas (\ref{15.3.7}) and $\Gamma(\alpha)\Gamma(1-\alpha)=\frac{\pi}{\sin \pi\alpha}$ were utilized. 
This expression continues $G_\rho$ analytically from $\comp^+\cap\{|z| < \delta\}$ to $\mathcal{S}_{(-\infty,\pi)}(\delta)$. 

Let $P:\mathcal{S}\to \comp\setminus\{0\}$ be the analytic map defined by $P(re^{i\theta}):=re^{i(\theta-2n\pi)}$ for $re^{i\theta}\in \mathcal{S}^{(n)}$.
We show that the Cauchy transform $G_{\rho'}$ also analytically extends from $\comp^+\cap\{|z| < \delta\}$ to $\mathcal{S}_{(-\infty,\pi)}(\delta)$, and we have the formula  
\begin{equation}\label{eq090} 
G_{\rho'}(z) =
\widetilde{G}_{\rho'}(P(z)) -2\pi c iz^{\alpha-1} \sum_{k=0}^{|n|-1} e^{2\pi \alpha k} f(re^{i(\theta+2\pi k)})
\end{equation}
for $z=re^{i\theta}\in \mathcal{S}^{(n)},~n \leq -1$. 
This is proved by the iterative application of the residue theorem as in Proposition \ref{prop78}. 
In the case $n=-1$, for $z\in \mathcal{S}_{(-\pi,0)}(\delta)$, we have the formula 
\begin{equation}\label{eq0901}
G_{\rho'}(z) =
\widetilde{G}_{\rho'}(P(z)) -2\pi c i z^{\alpha-1} f(re^{i\theta})
\end{equation}
along the same line as (\ref{eqbeta}).  Note now that $P(z)=z$ since $\mathcal{S}_{(-\pi,0)}(\delta) \subset \comp$. Clearly the expression (\ref{eq0901}) analytically extends from $\mathcal{S}_{(-\pi,0)}(\delta)$ to $\mathcal{S}_{(-2\pi,0)}(\delta)$ since $\widetilde{G}_{\rho'}(z)$ is analytic in $\mathcal{S}_{(-2\pi,0)}(\delta)$. Next we extend it from $\mathcal{S}_{(-2\pi,0)}(\delta)$ to $\mathcal{S}_{(-3\pi,0)}(\delta)$. This is done by taking a simple arc as in the proof of Proposition \ref{prop78}, and then we may extend $\widetilde{G}_{\rho'}(P(z))$ from $\mathcal{S}_{(-2\pi,0)}(\delta)$ to $\mathcal{S}_{(-3\pi,0)}(\delta)$. For $z\in\mathcal{S}_{(-3\pi,-2\pi)}(\delta)$, this analytic continuation of $\widetilde{G}_{\rho'}(P(z))$ is expressed by 
\begin{equation}\label{eq0902}
\widetilde{G}_{\rho'}(P(z)) - 2\pi c i  (P(z))^{\alpha-1} f(P(z))
\end{equation}
along the same line as (\ref{eqbeta}). 
Since now $P(z)=re^{i(\theta +2\pi)}$, combining (\ref{eq0901}) and (\ref{eq0902}) we get (\ref{eq090}) for $n=-2$. Iterating this argument, we get the general case (\ref{eq090}). 

\emph{Step 2: Asymptotic behavior of $G_\mu$ at $0$.} 
 We show that, as $|z| \to 0$,   
\begin{equation}\label{master}
G_\mu(z)= 
\begin{cases}
-\gamma(-z)^{\alpha-1}+o(|z|^{\alpha-1}), &z \in \mathcal{S}_{(-\theta(\alpha)-\theta_0/2, \pi)}(\delta), ~~ 0< \alpha<1, \\
\beta + \gamma(-z)^{\alpha-1}+o(|z|^{\alpha-1}), &z \in \mathcal{S}_{(-\theta(\alpha)-\theta_0/2, \pi)}(\delta),~~ 1< \alpha<2, 
\end{cases}
\end{equation}
where: $\beta \in \real, \gamma>0$ are real constants depending on $\alpha$; 
the function $(-z)^{\alpha-1}$ is extended to $\mathcal{S}$ so that it takes real values for $\arg z =\pi$; 
Landau's symbol $o(|z|^{\alpha-1})$ here means the uniform estimate in $ \mathcal{S}_{(-\theta(\alpha)-\theta_0/2, \pi)}$, that is,  
for any $\varepsilon>0$, there exists $\delta'\in(0,\delta)$ such that $|o(|z|^{\alpha-1})| \leq \varepsilon |z|^{\alpha-1}$ for $z\in \mathcal{S}_{(-\theta(\alpha)-\theta_0/2, \pi)}(\delta')$. 

This problem reduces to the study of $\rho'$ and $\rho''$ because $G_\rho$ is of the form (\ref{eq080}). 
We will show that the contribution of $\rho'+\rho''$ is the following: 
\begin{enumerate}[\rm(a)]
\item $G_{\rho'+\rho''}(z)= o(|z|^{\alpha-1})$ if $\alpha \in(0,1)$; 
\item $G_{\rho'+\rho''}(z)= \widetilde{\beta}+o(|z|^{\alpha-1})$ for some $\widetilde{\beta}\in\real$ if $\alpha \in(1,2)$. 
\end{enumerate}
Because $G_{\rho''}$ is analytic around $0$,  the above estimates are easy for $G_{\rho''}$. Hence we only have to estimate $G_{\rho'}.$ 
The summation part in (\ref{eq090}) is clearly of the form $o(|z|^{\alpha-1})$ because of the assumption (\ref{z4}), and so 
 it suffices to estimate $\widetilde{G}_{\rho'}$ in $\mathcal{S}_{(-2\pi,\xi)}$ for some $\xi>0$.

\textbf{Case $\alpha \in(0,1)$.} We divide the estimate into two cases $z \in \mathcal{S}_{[-\pi,\theta_0/2)}$ and $z \in \mathcal{S}_{(-2\pi,-\pi]}$. Take $\theta_0' \in (\frac{1}{2}\theta_0, \theta_0)$ arbitrarily. 
In Eq.\ (\ref{eq090}), displacing the contour $[0,\delta]$ of $\widetilde{G}_{\rho'}$ to $e^{i\theta_0'}[0,\delta] \cup \{\delta e^{i\theta}: \theta \in [0,\theta_0']\}$, one can write the Cauchy transform $\widetilde{G}_{\rho'}$ as\footnote{The minus sign of the second term is due to the direction of the curve.} 
\begin{equation}\label{eq070}
\begin{split}
\widetilde{G}_{\rho'}(z) 
&= c\int_{(0,\delta)e^{i\theta_0'}} \frac{1}{z-w} w^{\alpha-1}f(w)\,dw - c\int_{\{\delta e^{i\theta}: \theta \in [0,\theta_0']\}} \frac{1}{z-w} w^{\alpha-1}f(w)\,dw
\end{split}
\end{equation} 
 for $z \in \mathcal{S}_{[-\pi,\theta_0/2)}(\frac{\delta}{2})$. The second term is bounded for $z$ in a neighborhood of $0$, so it is $o(|z|^{\alpha-1})$. 
Writing $z=|z|e^{i\theta}$, $\theta \in [-\pi,\frac{1}{2}\theta_0)$, the first integral can be estimated as follows: 
\begin{equation}\label{eq072}
\begin{split}
\left|\int_{(0,\delta)e^{i\theta_0'}} \frac{1}{z-w} w^{\alpha-1}f(w)\,dw\right|
&=\left|\int_0^{\frac{\delta}{|z|}}\frac{|z|^{\alpha-1}e^{i(\alpha-1)\theta_0'}}{e^{i\theta}-e^{i\theta_0'}y} y^{\alpha-1}f(|z|ye^{i\theta_0'})ie^{i\theta_0'}\,dy\right| \\
&\leq |z|^{\alpha-1}\left|\int_0^{\frac{\delta'}{|z|}}\frac{1}{e^{i(\theta-\theta_0')}-y} y^{\alpha-1}f(|z|ye^{i\theta_0'})\,dy\right| \\
&~~~~+ |z|^{\alpha-1}\left|\int_{\frac{\delta'}{|z|}}^{\frac{\delta}{|z|}}\frac{1}{e^{i(\theta-\theta_0')}-y} y^{\alpha-1}f(|z|ye^{i\theta_0'})\,dy\right|
\end{split}
\end{equation}
for any $0 < \delta' < \delta/2$. Let $C>0$ be a constant such that $\left|\frac{1}{e^{i(\theta-\theta_0')}-y}\right| \leq \frac{C}{y+1}$ for any $\theta\in(-\pi, \theta_0/2)$ and any $y>0$. 
For an arbitrary $\varepsilon >0$, we can find $\delta'\in(0,\delta/2)$ such that $\sup_{w \in \mathcal{S}_{(-\pi, \theta_0')}(\delta')}|f(w)| < \varepsilon$, so that
\begin{equation}\label{eq0701}
\begin{split}
\left|\int_0^{\frac{\delta'}{|z|}}\frac{1}{e^{i(\theta-\theta_0')}-y} y^{\alpha-1}f(|z|ye^{i\theta_0'})\,dy\right| 
&\leq  CK\varepsilon,  
\end{split}
\end{equation}
where $K:=\int_0^\infty \frac{y^{\alpha-1}}{1+y}\,dy<\infty. $
Moreover, 
\begin{equation}\label{eq0702}
\begin{split}
\left|\int_{\frac{\delta'}{|z|}}^{\frac{\delta}{|z|}}\frac{1}{e^{i(\theta-\theta_0')}-y} y^{\alpha-1}f(|z|ye^{i\theta_0'})\,dy\right|
& \leq CM\int_{\frac{\delta'}{|z|}}^{\frac{\delta}{|z|}}\frac{y^{\alpha-1}}{1+y}\,dy,
\end{split}
\end{equation}
where $M:=\sup_{w \in \mathcal{S}_{(-\pi,\theta_0')}(\delta)}|f(w)|$.  If $\delta'' \in (0,\delta')$ is sufficiently small so that 
$\int_{\frac{\delta'}{|z|}}^{\frac{\delta}{|z|}}\frac{y^{\alpha-1}}{1+y}\,dy < \varepsilon$ for $|z| < \delta''$, we finally obtain 
\begin{equation}\label{eq0703}
\left|\int_{(0,\delta)e^{i\theta_0'}} \frac{1}{z-w} w^{\alpha-1}f(w)\,dw\right| \leq C(K+M)\varepsilon,~~|z|<\delta'', 
\end{equation}
which, together with (\ref{eq070}), shows that $\widetilde{G}_{\rho'}(z)$ is of order $o(|z|^{\alpha-1})$ for $\arg z \in [-\pi, \frac{1}{2}\theta_0')$. 

The case $z\in \mathcal{S}_{(-2\pi,-\pi]}(\frac{1}{2}\delta)$ is similar; we only have to displace the contour $(0,\delta)$ of $\widetilde{G}_{\rho'}$ to a curve in $\comp^-$. 

\textbf{Case $\alpha\in(1,2)$.} The second term of (\ref{eq070}) is analytic around $0$, so that one finds 
a constant $d \in \comp$ such that 
\begin{equation}\label{eq071}
c\int_{\{\delta e^{i\theta}: \theta \in [0,\theta_0']\}} \frac{1}{z-w} w^{\alpha-1}f(w)\,dw= d + O(z) = d+o(|z|^{\alpha-1})
\end{equation} 
for $z \in\mathcal{S}_{[-\pi, \frac{1}{2}\theta_0)}(\frac{1}{2}\delta)$. 
 By using the notation as in (\ref{eq072}), we have 
\begin{equation}\label{eq073}
\begin{split}
\int_{(0,\delta)e^{i\theta_0'}} \frac{1}{z-w} w^{\alpha-1}f(w)\,dw 
&=d'+\int_{(0,\delta)e^{i\theta_0'}} \frac{z}{z-w} w^{\alpha-2}f(w)\,dw\\
&=d'+ ie^{i(\alpha-1)\theta_0'}|z|^{\alpha-1}\int_0^{\frac{\delta}{|z|}}\frac{e^{i(\theta-\theta_0')}}{e^{i(\theta-\theta_0')}-y}y^{\alpha-2}f(|z|ye^{i\theta_0'})\,dy, 
\end{split}
\end{equation}
where 
$d':= -\int_{(0,\delta)e^{i\theta_0'}} w^{\alpha-2}f(w)\,dw.$ The second term  in (\ref{eq073}) is of order $o(|z|^{\alpha-1})$ from estimates like (\ref{eq0701})--(\ref{eq0703}). 
Therefore, we have obtained 
\begin{equation}\label{est1}
\widetilde{G}_{\rho'}(z)= \beta' +o(|z|^{\alpha-1}),~~\beta' \in\comp
\end{equation}
uniformly in the sector  $z \in \mathcal{S}_{[-\pi,\theta_0/2)}$. The case $z\in \mathcal{S}_{(-2\pi,-\pi]}(\frac{1}{2}\delta)$ is similar and so (\ref{est1}) holds in $\mathcal{S}_{(-2\pi, \theta_0/2)}$. The constant $\beta'$ is real because $\beta'=\widetilde{G}_{\rho'}(-0)$.

\emph{Step 3: Analysis of the compositional inverse of $F_\mu$.} Assuming $\mu$ to be FID, we deduce a contradiction. From Theorem \ref{thm0}, the map $F_\mu^{-1}$ analytically extends to $\comp^+$ via
\begin{equation}\label{inv}
F_\mu^{-1}(z):= \phi_\mu(z)+z,~~z\in\comp^+,  
\end{equation}
and hence $F_\mu$ is univalent in $\comp^+$ because $F^{-1}_\mu \circ F_\mu = \text{Id}$ in $\comp^+$ by identity theorem. So the map $F_\mu^{-1}$ defined in (\ref{inv}) is equal to the inverse map of $F_\mu|_{\comp^+}$ in the common domain $F_\mu(\comp^+).$ 

\textbf{Case $\alpha \in (1-\frac{1}{2n}, 1-\frac{1}{2n+1})$.} Note that this assumption is equivalent to $(2n-1)\pi <\theta(\alpha)<2n\pi$. First we neglect the perturbation $o(|z|^{\alpha-1})$ to get a rough picture of the idea. Let $\eta:= \frac{1}{2}\min\{2n\pi-\theta(\alpha),  \theta_0\}>0$. The map $F_\alpha(z):=-\gamma^{-1}(-z)^{1-\alpha}$ is a bijection from $\mathcal{S}_{(-\theta(\alpha) -\eta,\pi)}$ onto $\mathcal{S}_{(-(1-\alpha) \eta, \pi)}$, and hence we can define $F_\alpha^{-1}$ in $\mathcal{S}_{(-(1-\alpha) \eta, \pi)}$. Note that 
$\text{Im}\, F_\alpha (re^{-i\theta(\alpha)})=0$ for $r>0$, and that the point $re^{-i\theta(\alpha)}$ (as a point of $\comp$) is contained in $\comp^+$ because $-\theta(\alpha) \in(-2n\pi,-(2n-1)\pi)$.  Let $\phi_\alpha(z):=F_\alpha^{-1}(z)-z$. Then $\phi_\alpha(F_\alpha(re^{-i\theta(\alpha)})) =re^{-i\theta(\alpha)} - F_\alpha(re^{-i\theta(\alpha)})\in\comp^+$.  

Next, under the perturbation term $o(|z|^{\alpha-1})$, we prove the existence of a point $z_\alpha$ with angle close to $-\theta(\alpha)$ so that $z_\alpha\in \mathcal{S}_{(-2n\pi,-(2n-1)\pi)}$. Taking small $\eta_1,\eta_2, \eta_3>0$, we consider regions 
\[
U:= \mathcal{S}_{(-\theta(\alpha)-\eta, \pi)}(\eta_1), ~~~~~V:= \mathcal{S}_{(-\eta_2, \pi)}(\eta_3). 
\] 
Thanks to (\ref{master}), for sufficiently small $\eta_2,\eta_3 >0$, the boundary of $U$ is mapped by $F_\mu$ to a closed curve surrounding each point of $V$ exactly once. 
This implies that $F_\mu|_U$ takes each point of $V$ exactly once, and so  we can define $F_\mu|_U^{-1}$ in $V$.  Since $P$ is analytic, $P\circ F_\mu|_U^{-1}$ is also analytic in $V$, not necessarily univalent. Since $V$ intersects $F_\mu(\comp^+)$,  
the map $P\circ F_\mu|_U^{-1}$ coincides with the definition (\ref{inv}) in $V \cap \comp^+$ by the uniqueness of analytic continuation, and so 
$$
\phi_\mu(z) = P( F_\mu|_U^{-1}(z)) -z,~~z\in V\cap\comp^+,  
$$
which also extends $\phi_\mu$ to $V$ analytically. 

Take a $w_\alpha\in V \cap (0,\infty) \neq \emptyset$ and define $z_\alpha:= F_\mu|_U^{-1}(w_\alpha) \in U$. If $w_\alpha$ is close to $0$, then $\arg z_\alpha$ is close to $-\theta(\alpha)$ and so $z_\alpha\in \mathcal{S}_{(-2n\pi,-(2n-1)\pi)}$. Since $P(z_\alpha) \in \comp^+$,\footnote{Here the assumption $\alpha \in \mathcal{I}$ is used.} 
one finds that $\phi_\mu(w_\alpha) = P(z_\alpha)-w_\alpha \in \comp^+$, a contradiction to Theorem \ref{thm0}. 

Case $\alpha \in (1+\frac{1}{2n+1}, 1+\frac{1}{2n})$ is similar. 
%Note that the set $V$ is to be replaced by a $\mathcal{S}_{(-\eta_2, \pi)}\cap\{re^{i\theta}\in \mathcal{S}: r>R\}$ with large $R>0$.  
\end{proof}

\begin{exa} 
\begin{enumerate}[\rm(i)]
\item If $X \sim \bm{\beta}_{p,q}$ for $q \in\mathcal{I}$, then the law of $X^r$ is not FID for any $r \in \real \setminus\{0\}$. 
The density function of the law of $X^r$ is given by 
\[
\begin{cases}
\frac{1}{r B(p,q)}x^{\frac{p}{r}-1}(1-x^{\frac{1}{r}})^{q-1}, &0<x<1, ~r>0,\\ 
\frac{1}{|r| B(p,q)}x^{-\frac{p+q-1}{|r|}-1}(x^{\frac{1}{|r|}}-1)^{q-1}, &x>1, ~r<0,    
\end{cases}
\]
which behaves as $c |x-1|^{q-1}$ around $x=1$. 

\item The standard semicircle law $\mathbf{w}$, at $x=\pm2$, corresponds to $\alpha=\frac{3}{2}$ which is in the closure of $\mathcal{I}$, but $\mathbf{w}$ is FID. 

\item If $X \sim \mathbf{w}$, then $|X|^p \sim |\mathbf{w}|^p:=\frac{1}{\pi p}x^{\frac{1}{p}-1}\sqrt{4-x^{\frac{2}{p}}}\,1_{[0,2^p]}(x)\,dx$. Looking at the density function around $x=0$, the law $|\mathbf{w}|^p$ is not FID for $p\in \mathcal{I}^{-1}=\left(\bigcup_{n=1}^\infty(1-\frac{1}{2n+1}, 1-\frac{1}{2n+2})\right)\cup\left(\bigcup_{n=1}^\infty(1+\frac{1}{2n}, 1+\frac{1}{2n-1})\right)$. Note that $|\mathbf{w}|^p$ is FID for $p=2,4$ \cite[Example 7.4]{AHS}. 
\end{enumerate}
\end{exa}

\subsection{Method based on subordination function} 
We utilize subordination functions introduced by Voiculescu \cite{V93}, in order to show the following. 
\begin{prop}
 $\bm{\beta}_{p,q}$ is not FID for $0<p,q \leq 1$. 
\end{prop}
\begin{proof}

Let $\mu$ be FID and $\mu_t:=\mu^{\boxplus t}$. 
For $s \leq t$, a function $\omega_{s,t}: \mathbb{C}^+\to\mathbb{C}^+$ exists so that it satisfies $F_{\mu_s} \circ \omega_{s,t} = F_{\mu_t}$. The map $\omega_{s,t}$ is called the \emph{subordination function} for $(\mu_t)_{t \geq 0}$. We can write $\omega_{s,t}$ in terms of $F_{\mu_t}$: 
\begin{equation}\label{eq1}
F_{\mu_t}(z)= \frac{t/s}{t/s-1}\omega_{s,t}(z)-\frac{z}{t/s-1}.
\end{equation} 
It is proved in Theorem 4.6 of \cite{BB05} that $\omega_{s,t}$ and hence 
$F_{\mu_t}$ extends to a continuous function from $\mathbb{C}^+\cup \real$ into itself. 
Moreover $\omega_{s,t}$ satisfies the inequality 
$$
|\omega_{s,t}(z_1)-\omega_{s,t}(z_2)| \geq \frac{1}{2}|z_1 - z_2|, ~~~z_1,z_2 \in \mathbb{C}^+ \cup \real. 
$$ 
Taking the limit $s \to 0$ in (\ref{eq1}), we get 
$$
|F_{\mu_t}(z_1)-F_{\mu_t}(z_2)| \geq \frac{1}{2}|z_1 - z_2|, ~~~z_1,z_2 \in \mathbb{C}^+ \cup \real, 
$$ 
so that $F_{\mu_t}$ is injective on $\comp^+\cup \real$. 
 
For $p,q \in (0,1]$, the density of $\bm{\beta}_{p,q}$ is not continuous at two points $0,1$, so that its 
reciprocal Cauchy transform is zero at $z=0,1$, which implies that the measure is not FID. 
\end{proof}

\subsection{Hankel determinants of free cumulants for $p=1$ or $q=1$}\label{Hankel} Instead of the analytic method, one can also compute \emph{free cumulants} $(r_n)_{n\geq 1}$ to show that a measure is not FID. The reader is referred to \cite{BG06} and \cite{NS06} for information on free cumulants. 
%If $\mu$ is a probability measure with finite $k$th moments, then for any $\alpha>0$ there is $M>0$ such that  
%$$
%\phi_\mu(z)= \sum_{n=1}^k\frac{r_n}{z^{n-1}} +o(z^{-(k-1)}), ~~z\to \infty, z\in \Gamma_{\alpha,M};  
%$$
%see \cite{B05}. The coefficients $r_n$ are called the \emph{free cumulants} of $\mu.$ 
%If the Hankel determinant $\text{det}(r_{i+j})_{1 \leq i,j \leq m}$ is negative for some $m \leq k$, then 
%$\mu$ is not FID.   

The exponential distribution is the limit of $D_{q}\bm{\beta}_{1, q}$ as $q \to\infty$. It is not FID since the 16th Hankel determinant 
$
\begin{vmatrix}
r_2 &r_3& r_4 & \cdots &r_{17}\\
r_3&r_4&r_5&\cdots & r_{18} \\
&\cdots&\cdots&\cdots&\\
r_{17} &r_{18}&r_{19}&\cdots&r_{32}
\end{vmatrix}
$ 
of $(r_n)_{n \geq 2}$
 is negative. This implies that $\bm{\beta}_{1,q}$ is not FID for large $q>0$, because the set of non FID distributions is open with respect to the weak convergence. 
 For smaller $q>0$, $\bm{\beta}_{1,q}$ is still not FID; they have negative Hankel determinants for $q=1,2,\cdots,15$.
 %\footnote{${}^{,5}$The Hankel determinants were computed by Mathematica. \label{aaa}} 
 
 The beta prime distribution $\bm{\beta}'_{1, q}$ is called the Pareto distribution. With suitable scaling, they also converge to the exponential distribution as $q \to \infty$, so that $\bm{\beta}'_{1,q}$ is not FID for large $q$. Actually $\bm{\beta}'_{1,q}$ is not FID for $q=60,61,62, 70, 90,100, 150$ because their 26th, 25th, 24th, 21th, 18th, 18th, 16th Hankel determinants are negative respectively.
 
Thus $\bm{\beta}_{1,q}$ and $\bm{\beta}_{1,q}'$ are not FID for many parameters $q>0$. 
Recalling that $\bm{\beta}_{p,2-p}$ is not FID for $\frac{1}{2}<p<\frac{3}{2}$ \cite[Theorem 5.5]{AHb} and $\bm{\beta}_{p,p}$ is not FID for $0<p<\frac{3}{2}$ \cite[Corollary 4.1]{AP10}, 
the author poses the following conjectures. 
\begin{conj}
\begin{enumerate}[(1)]

\item $\bm{\beta}_{p,q}$ and $\bm{\beta}_{p,q}'$ are not FID for $ p \in (\frac{1}{2},\frac{3}{2}), q>0$. 

\item\label{conj2}$\bm{\beta}_{p,q}$ is FID if $(p,q) \in(0, \frac{1}{2}] \times [\frac{3}{2}, 2]$ or if $(p,q) \in  [\frac{3}{2},2]\times (0,\frac{1}{2}]$. 
\item $\bm{\gamma}_p$ is not FID for $p \in(\frac{1}{2},\frac{3}{2}).$
\end{enumerate}
\end{conj}
The conjecture (\ref{conj2}) seems to be true from numerical computation. This case however is not covered by Theorem \ref{thm11} because the assumption $p+q > 2$ is crucial to prove Lemma \ref{lem3}.  

One may expect that the proof of Theorem \ref{thm90} also applies to any $\alpha \in (\frac{1}{2}, \frac{3}{2})$, but just a slight modification seems not sufficient for that purpose. 
\begin{prob}
Does Theorem \ref{thm90} extend to arbitrary $\alpha \in (\frac{1}{2}, \frac{3}{2})$? 
\end{prob}

\section{Free infinite divisibility for Student t-distribution}\label{st}

We are going to utilize Proposition \ref{prop91}(2) to prove that t-distributions are FID. 

\begin{prop}\label{prop32}
The Cauchy transform $G_{\mathbf{t}_q}$ analytically extends  to the domain $\mathcal{D}^{st}:= (\comp^- \cup \mathbb{H}^+)\setminus i[-1,0]$. We denote the analytic continuation by $G_{\mathbf{t}_q}$ too. Then 
\[
G_{\mathbf{t}_q}(z)=
\widetilde{G}_{\mathbf{t}_q}(z) -\frac{2\pi i}{B(\frac{1}{2}, q-\frac{1}{2})} (1+z^2)^{-q},~~ z\in \comp^{-}\setminus i[-1,0),
\]
where $(1+z^2)^{-q}$ is defined analytically in $\mathcal{D}^{st}$ so that $(1+x^2)^{-q} \in \real$ for $x >0$. 
\end{prop}
\begin{proof}
 The proof is quite similar to that of Proposition \ref{prop78}. 
 \end{proof}

For condition (\ref{diffB}), its proof is based on a recursive differential equation that is quite similar to Lemma \ref{dif}. 
\begin{lem} \label{lem11}
\begin{enumerate}[\rm(1)]
\item \label{st1} $\widetilde{G}_{\mathbf{t}_{q+1}}(z) = \frac{1}{1+z^2}\left( \frac{q}{q-1/2}\widetilde{G}_{\mathbf{t}_q}(z) +z \right), \quad z\in\comp^+$. 

\item \label{lem31}$\frac{d}{dz} \widetilde{G}_{\mathbf{ t}_{q}}(z) = (2q-1)\left(1 -z \widetilde{G}_{\mathbf{ t}_{q+1}}(z)\right) = \frac{2q}{1+z^2}\left(\frac{q-1/2}{q}-z\widetilde{G}_{\mathbf{t}_{q}}(z)\right),\quad z\in\comp^+$.  
\end{enumerate}
\end{lem}
\begin{proof}

(\ref{st1})\,\,\, Let $c_q:=\frac{1}{B(\frac{1}{2},q-\frac{1}{2})}$. 
From simple calculations,   
\[
\begin{split}
\widetilde{G}_{\mathbf{t}_q}(z) 
&= c_q\int_{\real} \frac{1+z^2 -2z(z-x) + (z-x)^2}{(1+x^2)^{q+1}} \frac{1}{z-x}dx \\
&=\frac{c_q}{c_{q+1}} \widetilde{G}_{\mathbf{t}_{q+1}}(z) (1+z^2) - 2z \frac{c_q}{c_{q+1}} + z \frac{c_q}{c_{q+1}} \\
&= \frac{c_q}{c_{q+1}}\left( (z^2+1)\widetilde{G}_{\mathbf{t}_{q+1}}(z)-z \right). 
\end{split}
\]
The conclusion then follows since $\frac{c_q}{c_{q+1}}=\frac{q-1/2}{q}$. 

(\ref{lem31})\,\,\, By integration by parts,  
\[
\begin{split}
\widetilde{G}_{\mathbf{t}_{q}}'(z) 
&= -c_q\int_{\real} \frac{1}{(1+x^2)^q} \frac{1}{(z-x)^2}dx = -c_q\int_{\real} \frac{2qx}{(1+x^2)^{q+1}} \frac{1}{z-x}dx \\
&=2qc_q\int_{\real} \frac{z-x-z}{(1+x^2)^{q+1}} \frac{1}{z-x}dx= 2q\frac{c_q}{c_{q+1}} \left( 1 - z\widetilde{G}_{\mathbf{t}_{q+1}}(z)\right). 
\end{split}
\]
The second equality follows from (\ref{st1}). 
\end{proof}

By using Lemma \ref{lem11}, condition (\ref{diffB}) can be proved as in Lemma \ref{lem3}. 

\begin{lem}\label{lem81}
The Cauchy transform $G_{\mathbf{t}_q}$ satisfies condition (\ref{diffB}) for any $q > \frac{1}{2}$ in the domain $\mathcal{D}^{st}$. 
\end{lem}
\begin{proof}
Assume that $G_{\mathbf{t}_q}(z) \in \comp^-$, $G_{\mathbf{t}_q}'(z)=0$ and $z\in\mathcal{D}^{st}$. If $z\in\comp^+$, then $F_{\mathbf{t}_{q+1}}(z) = z$ from Lemma \ref{lem11}(\ref{lem31}), which contradicts Proposition \ref{prop08}(\ref{p2}).  If  $z \in \comp^- \cup \real$, then $G_{\mathbf{t}_q}(z) = \frac{q-1/2}{qz}$ from Lemma \ref{lem11}(\ref{lem31}), which contradicts the assumption  $G_{\mathbf{t}_q}(z) \in \comp^-$. 
\end{proof}

Moreover, the following property is required in (\ref{B}). 
\begin{lem}\label{lem129}
The Cauchy transform $G_{\mathbf{t}_q}$ extends to a univalent map around $i(-1,\infty)$ and it maps $i(-1,\infty)$ onto $i(-\infty,0)$. 
\end{lem} 
\begin{proof}
The differential equation of Lemma \ref{lem11}(\ref{lem31}) implies $\widetilde{G}_{\mathbf{t}_q}'(iy)= (2q-1)\int_{\real}\frac{x^2}{x^2+y^2}\mathbf{t}_{q+1}(dx)>0$ for $y \neq 0$ and $G'_{\mathbf{t}_q}(0)=\lim_{y \searrow 0}\widetilde{G}_{\mathbf{t}_q}'(iy) =2q-1 >0$. From Proposition \ref{prop32}, one has
$$
G_{\mathbf{t}_q}'(iy)= 
\begin{cases}
\widetilde{G}_{\mathbf{t}_q}'(iy) >0,& y \geq 0, \\[10pt]
\widetilde{G}_{\mathbf{t}_q}'(iy) -\dfrac{4\pi q y}{B(\frac{1}{2}, q-\frac{1}{2})}\left(1-y^2\right)^{-q-1} >0,& y\in(-1,0),   
\end{cases}
$$
and so $G_{\mathbf{t}_q}$ extends to a univalent map around $i(-1,\infty)$. Since $\lim_{y \searrow -1}\frac{1}{i}G_{\mathbf{t}_q}(iy)=-\infty$ from Proposition \ref{prop32} and $\lim_{y \to \infty}G_{\mathbf{t}_q}(iy)=0$, 
$G_{\mathbf{t}_q}$ maps $i(-1,\infty)$ onto $i(-\infty,0).$
\end{proof}
\begin{rem}[Simple proof of the free infinite divisibility of Gaussian]\label{gauss}
If $q=n\in\{1,2,3,\ldots\}$, then the Cauchy transform is a rational function: for $n=1$, the measure is a Cauchy distribution and 
$G_{\mathbf{t}_1}(z)=\frac{1}{z+i}$, and then from Lemma \ref{lem11}(\ref{st1}) one can recursively show the claim. 
One can show condition (\ref{diffA}) for the domain $\comp$ similarly to the proof of Lemma \ref{lem81}. 
Since the Cauchy transform is rational and $\lim_{y\to\infty}iyG_{\mathbf{t}_n}(iy)=1$, condition (\ref{barrierA}) is easily verified: if $z_k\in\comp$ converge to infinity, then $G_{\mathbf{t}_n}(z_k) \to 0 \in \real$ as $k \to \infty$. Therefore $G_{\mathbf{t}_n}$ satisfies (\ref{A}) and hence $\mathbf{t}_n$ is FID. After taking a limit $n\to\infty$ with some scaling, we have that the Gaussian is also FID. For general $q$, we need much more computation to verify condition (\ref{barrierB}), which is shown in Theorem \ref{thm31} below.
\end{rem}

Now we are going to complete the proof of Theorem \ref{thm1}. The construction of curves as in beta and beta prime distributions is now difficult, and instead we more directly apply Proposition \ref{prop91}(2).  We utilize two domains $\mathbb{H}^+$ and $\mathcal{D}^{st}$ to get a better result, but still the result does not cover some exceptional parameters $q$. 

\begin{thm}\label{thm31}
\begin{enumerate}[\rm(1)]
\item\label{thm31(2)} The t-distribution $\mathbf{t}_q$ is in the class $\mathcal{UI}$ provided $q \in (\frac{1}{2},2] \cup \bigcup_{n=1}^\infty [2n + \frac{1}{2}, 2n+2]$.
 
\item The t-distribution $\mathbf{t}_q$ is in the class $\mathcal{UI}_s$ provided $q \in (\frac{1}{2},1] \cup \bigcup_{n=1}^\infty [n + \frac{1}{4}, n+1]$. 
\end{enumerate}
\end{thm}
\begin{proof}
(1)\,\,\, Consider $\mathbb{H}^+$ as a domain required in condition (\ref{B}). Conditions (\ref{merB}) and (\ref{diffB}) are respectively shown in Proposition \ref{prop32} and Lemma \ref{lem81}, and the first part of condition ({\ref{B}}) is proved in Lemma \ref{lem129}. To show the remaining condition (\ref{barrierB}),  
the following values have to belong to $\overline{\mathbb{H}^- \cup \comp^+}  \cup \{\infty\}$: 
\begin{align}
&G_{\mathbf{t}_{q}}(+0+iy),~~ y \in \real \setminus \{-1\}, \label{eq70}\\
&\lim_{z \to -i, z \in \mathbb{H}^+} G_{\mathbf{t}_q}(z), \label{eq71} \\ 
& \lim_{z \to \infty, z\in\mathbb{H}^+}G_{\mathbf{t}_{q}}(z). \label{eq72} 
\end{align}

(\ref{eq71})\,\,\, It is easy to see $\lim_{z \to -i, z \in \mathbb{H}^+} G_{\mathbf{t}_q}(z)=\infty$ from Proposition \ref{prop32}.

 (\ref{eq72})\,\,\, $G_{\mathbf{t}_q}(z)$ can be written as $\int_{\real- \frac{i}{2}}\frac{1}{z-x} \mathbf{t}_q(dx)$, and so $\lim_{z\to \infty, z\in \comp^+}G_{\mathbf{t}_q}(z) =0$. From Proposition \ref{prop32}, we also deduce that  $\lim_{z\to \infty, z\in \mathbb{H}^+}G_{\mathbf{t}_q}(z) =0$.

(\ref{eq70})\,\,\, Lemma \ref{lem129} implies that $G_{\mathbf{t}_{q}}(+0+iy) \in i \real$ for $y>-1$.  The same proposition enables us to calculate the boundary values for $y<-1$ as follows: 
\[
\begin{split}
&\text{Re}\,G_{\mathbf{t}_q}(+0+iy)=-\text{Re}\left(\frac{2\pi i}{B(\frac{1}{2}, q-\frac{1}{2})} (y^2-1)^{-q}e^{\pi q i}\right)=\frac{2\pi \sin(\pi q)}{B(\frac{1}{2}, q-\frac{1}{2}) (y^2-1)^{q}},   \\
&\text{Im}\,G_{\mathbf{t}_q}(+0+iy)=\text{Im}\, \widetilde{G}_{\mathbf{t}_q}(iy) - \text{Im}\left(\frac{2\pi i}{B(\frac{1}{2}, q-\frac{1}{2})} (y^2-1)^{-q}e^{\pi q i}\right)\\
&~~~~~~~~~~~~~~~~~~~~~=\text{Im}\, \widetilde{G}_{\mathbf{t}_q}(iy) - \frac{2\pi \cos(\pi q)}{B(\frac{1}{2}, q-\frac{1}{2}) (y^2-1)^{q}}.  
\end{split}
\]
Note that $\widetilde{G}_{\mathbf{t}_q}(z) \in \comp^+$ for $z \in \comp^-$. 
 The condition $q\in(\frac{1}{2},2] \cup \bigcup_{n=1}^\infty [2n + \frac{1}{2}, 2n+2]$ guarantees the inequalities $\text{Re}\,G_{\mathbf{t}_q}(+0+iy) \leq 0$ or $\text{Im}\,G_{\mathbf{t}_q}(+0+iy) \geq 0$ in $(-\infty,-1)$.

 (2)\,\,\, We are going to show condition (\ref{B}) for $\mathcal{D}^{st}$. 
The most important condition is (\ref{barrierB}); the others  can be shown similarly to the case (1). We show that the limiting values 
\begin{align}
&G_{\mathbf{t}_q}(x- i0),~x\in(-\infty,0],\label{eq73}\\
& G_{\mathbf{t}_q}(\pm 0+iy),~~y\in(-1,0), \label{eq74}\\
& \lim_{z \to -i, z\in\mathcal{D}^{st}}G_{\mathbf{t}_{q}}(z), \label{eq75} \\
& \lim_{z \to \infty, z\in\mathcal{D}^{st}}G_{\mathbf{t}_{q}}(z) \label{eq76} 
\end{align}
are all in $\overline{\mathbb{H}^- \cup \comp^+}  \cup \{\infty\}$. 

 (\ref{eq75}) and (\ref{eq76}) are computed as in the case (1) and they belong to  $\overline{\mathbb{H}^- \cup \comp^+}  \cup \{\infty\}$. 

(\ref{eq73})\,\,\, For $x\leq0$, the following computation holds: 
\begin{equation*}
\begin{split}
&\text{Re}\, G_{\mathbf{t}_q}(x-i0) = \text{Re}\, \widetilde{G}_{\mathbf{t}_q}(x-i0)- \text{Re}\left(\frac{2\pi i}{B(\frac{1}{2}, q-\frac{1}{2})} (1+x^2)^{-q}e^{2\pi q i}\right)\\
&~~~~~~~~~~~~~~~~~~~= \text{Re}\, \widetilde{G}_{\mathbf{t}_q}(x-i0)+ \frac{2\pi \sin(2\pi q)}{B(\frac{1}{2}, q-\frac{1}{2}) (1+x^2)^{q}},\\
&\text{Im}\, G_{\mathbf{t}_q}(x-i0)
=  \text{Im}\, \widetilde{G}_{\mathbf{t}_q}(x-i0)- \text{Im}\left(\frac{2\pi i}{B(\frac{1}{2}, q-\frac{1}{2})} (1+x^2)^{-q}e^{2\pi q i}\right)\\
&~~~~~~~~~~~~~~~~~~~=  \frac{\pi}{B(\frac{1}{2}, q-\frac{1}{2})} (1+x^2)^{-q}   - \frac{2\pi \cos(2\pi q)}{B(\frac{1}{2}, q-\frac{1}{2})} (1+x^2)^{-q}\\
&~~~~~~~~~~~~~~~~~~~=  \frac{\pi (3-4\cos^2(\pi q))}{B(\frac{1}{2}, q-\frac{1}{2}) (1+x^2)^{q}}, 
\end{split}
\end{equation*}
where the Stieltjes inversion formula was used to calculate $\text{Im}\, \widetilde{G}_{\mathbf{t}_q}(x-i0)$. Note that $\text{Re}\, \widetilde{G}_{\mathbf{t}_q}(x-i0) \leq 0$; see Lemma \ref{lem03}.  We need the inequality $\text{Re}\, G_{\mathbf{t}_q}(x-i0) \leq 0$ or $\text{Im}\, G_{\mathbf{t}_q}(x-i0) \geq 0$ on $(-\infty,0)$, which is true if $q \in \mathcal{B}_1:=(\frac{1}{2}, 1] \cup [\frac{7}{6}, 2] \cup [\frac{13}{6}, 3] \cup [\frac{19}{6}, 4] \cup \cdots$. 

(\ref{eq74})\,\,\, As in the case (1), one easily see that $G_{\mathbf{t}_q}(+0+iy) \in i\real$ for $y \in (-1,0]$. 
The other limit is computed as follows: for $-1<y<0$, 
\begin{equation*}
\begin{split}
&\text{Re}\, G_{\mathbf{t}_q}(-0+iy) =  -\text{Re}\left(\frac{2\pi i}{B(\frac{1}{2}, q-\frac{1}{2})} (1-y^2)^{-q}e^{2\pi q i}\right)\\
&~~~~~~~~~~~~~~~~~~~~~=  \frac{2\pi \sin(2\pi q)}{B(\frac{1}{2}, q-\frac{1}{2}) (1-y^2)^{q}},\\
&\text{Im}\, G_{\mathbf{t}_q}(-0+iy)
=  \text{Im}\, \widetilde{G}_{\mathbf{t}_q}(iy)- \text{Im}\left(\frac{2\pi i}{B(\frac{1}{2}, q-\frac{1}{2})} (1-y^2)^{-q}e^{2\pi q i}\right)\\
&~~~~~~~~~~~~~~~~~~~~~=\text{Im}\, \widetilde{G}_{\mathbf{t}_q}(iy)  - \frac{2\pi \cos(2\pi q)}{B(\frac{1}{2}, q-\frac{1}{2}) (1-y^2)^{q}}.
\end{split}
\end{equation*}
We need $\text{Re}\, G_{\mathbf{t}_q}(-0+iy) \leq 0$ or $\text{Im}\, G_{\mathbf{t}_q}(-0+iy) \geq 0$ for $-1<y <0$, which is true if  $q \in \mathcal{B}_2:= (\frac{1}{2}, 1] \cup [\frac{5}{4}, 2] \cup [\frac{9}{4}, 3] \cup [\frac{13}{4}, 4] \cup \cdots$. Consequently, condition (\ref{barrierB}) holds provided $q \in \mathcal{B}_1 \cap \mathcal{B}_2 = (\frac{1}{2},1] \cup \bigcup_{n=1}^\infty [n + \frac{1}{4}, n+1]$.
 \end{proof}

Looking at  the component of $\{z\in\comp: G_{\mathbf{t}_q}(z) \in \comp^- \}$ containing $\comp^+$ drawn by Mathematica, the following conjecture is likely to hold. 
\begin{conj}
The t-distribution $\mathbf{t}_q$ is FID (and more strongly in class $\mathcal{UI}$) for any $q>\frac{1}{2}$. 
\end{conj}

\section{The free divisibility indicator of symmetric FID distributions}\label{fdi}
A family of maps $\{\mathbb{B}_t \}_{t \geq 0}$ is defined  on the set of Borel probability measures $\mathcal{P}$ \cite{BN08}:  
\[
\mathbb{B}_t(\mu) := \Big(\mu^{\boxplus (1+t)}\Big) ^{\uplus\frac{1}{1+t}},~~t\geq0, ~~\mu\in \mathcal{P}, 
\]
where $\uplus$ is \emph{Boolean convolution} \cite{SW97} and the probability measure $\mu^{\uplus t}$ ($t\geq 0$) is defined by $F_{\mu^{\uplus t}}(z) = (1-t)z + tF_\mu(z)$. 
These maps become a flow: $\mathbb{B}_{t+s} = \mathbb{B}_t \circ \mathbb{B}_s$ for $s,t \geq 0$. 
The \emph{free divisibility indicator} $\phi(\mu) \in [0,\infty]$ is defined by 
\[
\phi(\mu):=\sup \{t \geq 0: \mu \in \mathbb{B}_t(\mathcal{P}) \}. 
\]
A probability measure $\mu$ is FID if and only if $\phi(\mu) \geq 1$ \cite{BN08}. The following property is known \cite{AHc}: 
$$ 
\phi(\mu^{\uplus t})=\frac{\phi(\mu)}{t},~t>0,~~~\phi(\mu) = \sup\{ t \geq0: \mu^{\uplus t} \text{~is~FID}\}.
$$ 
Hence, when $\phi(\mu) < \infty$, $\mu^{\uplus t}$ is FID  for small $t>0$, and the free divisibility indicator measures the time when the Boolean time evolution breaks the free infinite divisibility. 

We will give a method for calculating the quantity $\phi(\mu)$.

\begin{lem}\label{prop323}
Let $\mu$ be a symmetric FID distribution satisfying one of the following conditions:
\begin{enumerate}[\rm(1)]
\item\label{fdi1} $F_\mu$ extends to a univalent function around $i\real$ and it maps $i\real$ onto $i(0,\infty)$; 
\item\label{fdi2} $F_\mu$ extends to a univalent function around $i(c,\infty)$  for some $c\in(-\infty,0]$ and it maps $i(c,\infty)$ onto $i(0,\infty)$, and moreover $F_\mu'(ic+i0)=0$. 
\end{enumerate}
Then the free divisibility indicator of $\mu$ is 1. 
\end{lem}
\begin{proof}
For $t >1$, let $f(y):=\frac{1}{i}F_\mu(iy)$ and $f_t(y):=\frac{1}{i}F_{\mu^{\uplus t}}(iy) =  (1-t)y + t f(y)$. Suppose (1), then $f(-\infty) = 0$ and $f(\infty)=\infty$, and so $f_t(\pm\infty) =\infty$. 
Since $f_t'(\infty)=1$, we can find  a point $y_0 \in \real$ such that $f_t'(y_0) =0$ and $f_t'(y) >0$ for  $y \in (y_0,\infty)$. Let $y_1:= f_t(y_0) = (1-t)y_0 +tf(y_0)$. If $y_0  \leq 0$, then $y_1 > 0$ because $f>0$. If $y_0 >0$, then $y_1 \geq (1-t)y_0 + ty_0 = y_0 >0$ from Proposition \ref{prop08}(\ref{p2}). Hence $y_1>0$ in both cases. The inverse map $F_{\mu^{\uplus t}}^{-1}$ analytically extends to a neighborhood of $i(y_1,\infty)$, but $(F_{\mu^{\uplus t}}^{-1})'(iy_1 + i0) = \infty$. From Theorem \ref{thm0}, $\mu^{\uplus t}$ is not FID.

Suppose now (2), then $f_t'({}c)=1-t <0$ and $f_t'(\infty)=1$ and so we can find a point $y_0$ similarly. The remaining proof is the same as above. 
\end{proof}
\begin{prop}\label{cor0101} The free divisibility indicators of  Student t- and ultrashperical distributions can be calculated as follows. 
\begin{enumerate}[\rm(1)]
\item $\phi(\mathbf{t}_q)=1$ for $q \in (1, 2] \cup  \bigcup_{n=1}^\infty [2n + \frac{1}{4}, 2n+2]$.  
\item $\phi(\mathbf{u}_p)=1$ for $p \in [1,\infty)$. 
\end{enumerate}
\end{prop}
\begin{rem}
$\mathbf{t}_1$ is a Cauchy distribution and its free divisibility indicator is infinity because $\mathbf{t}_1^{\uplus t}$ is a Cauchy distribution too. The exact value is unknown for $q \in (\frac{1}{2},1)$. 

The measure $\mathbf{u}_{0}$ is a symmetric arcsine law and $\mathbf{u}_{-1/2}:=\lim_{p\to-1/2}\mathbf{u}_{p}$ is a symmetric Bernoulli law. It is known that $\phi(\mathbf{u}_0)=\frac{1}{2}$ and $\phi(\mathbf{u}_{-1/2})=0$ \cite[Table 1]{BN08}. The value $\phi(\mathbf{u}_p)$ is not known for $p \in (-\frac{1}{2},0)\cup (0,1)$. 
\end{rem}
\begin{proof}

(1)\,\,\,
$F_{\mathbf{t}_q}$ is univalent around $i(-1,\infty)$ and it maps $i(-1,\infty)$ onto $i(0,\infty)$ from Lemma \ref{lem129}. From Proposition \ref{prop32}, $F'_{\mathbf{t}_q}(iy)$ is approximately proportional to $(1-y^2)^{q-1}$ as $y \searrow -1$, so that $F'_{\mathbf{t}_q}(-i+i0) =0$ if (and only if) $q >1.$ Now Lemma \ref{prop323}(2) is applicable.

(2)\,\,\, Let $p\geq 1$. We are going to show that $\text{Im}\, G_{\mathbf{u}_p}(iy)$ is strictly increasing in $\real$, following Lemma \ref{lem129}.  
If $X \sim \bm{\beta}_{p+\frac{1}{2}, p+\frac{1}{2}}$, then $2X-1 \sim \mathbf{u}_p$. From Lemma \ref{lem2}(\ref{affine}) and Proposition \ref{prop78}, we have $G_{\mathbf{u}_p}(z)=\frac{1}{2}G_{\bm{\beta}_{p+\frac{1}{2}, p+\frac{1}{2}}}(\frac{z+1}{2})$ and 
\begin{equation}\label{eq0500}
G_{\mathbf{u}_p}(iy)=\widetilde{G}_{\mathbf{u}_p}(iy) - \frac{4^{-p+\frac{1}{2}}\pi i }{B(p+\frac{1}{2},p+\frac{1}{2})}(1+y^2)^{p-\frac{1}{2}}, ~~y\in(-\infty,0),  
\end{equation}
and so 
\begin{equation}\label{eq0505}
G'_{\mathbf{u}_p}(iy)=\widetilde{G}'_{\mathbf{u}_p}(iy) - \frac{4^{-p+\frac{1}{2}}\pi(2p-1)y}{B(p+\frac{1}{2},p+\frac{1}{2})}(1+y^2)^{p-\frac{3}{2}}, ~~y\in(-\infty,0). 
\end{equation}
The second term is positive. 

The differential equation in Lemma \ref{dif} yields 
\begin{equation}\label{eq0501}
\begin{split}
\widetilde{G}'_{\mathbf{u}_p}(iy) &= \frac{1}{1+y^2} +\frac{2p-1}{1+y^2}(1-iy\widetilde{G}_{\mathbf{u}_p}(iy)) \\
&= \frac{1}{1+y^2} +\frac{2p-1}{1+y^2}\int_{-1}^1\frac{x^2}{x^2+y^2}\,\mathbf{u}_p(dx) >0,~~y \neq 0, 
\end{split}
\end{equation}
and also $G_{\mathbf{u}_p}'(0)=2p>0$. (\ref{eq0505}) and (\ref{eq0501}) entail that $G_{\mathbf{u}_p}'(iy)>0$ for $y\in\real$. 
Moreover, $G_{\mathbf{u}_p}(z)$ maps $i\real$ onto $i(-\infty,0)$ because $\text{Im}\, G_{\mathbf{u}_p}(iy) \to 0$ as $y \to \infty$ and $\text{Im}\, G_{\mathbf{u}_p}(iy) \to -\infty$ as $y \to -\infty$ because of (\ref{eq0500}). Now we can apply Lemma \ref{prop323}(1). 
\end{proof}

The free divisibility indicator is not continuous with respect to the weak convergence, as one can observe from Wigner's semicircle law $\mathbf{w}_t$ with mean $0$ and variance $t$. Indeed, $\phi(\mathbf{w}_t)=1$ for any $t >0$, while 
$\phi({\mathbf{w}_0}) = \infty$ (see \cite{BN08} for this computation). Hence, Proposition \ref{cor0101}  is not sufficient to calculate the exact value of the free divisibility indicator of Gaussian which is the weak limit of scaled ultraspherical distributions or t-distributions. Here we will show that the value is exactly 1 for the Gaussian distribution. The classical infinite divisibility of the Boolean power of Gaussian is also studied here.

\begin{prop}\label{prop90} Let $\mathrm{\bf g}$ be the standard Gaussian. 
\begin{enumerate}[\rm(1)]
\item $\phi(\mathrm{\bf g})=1$, or equivalently, $\mathrm{\bf g}^{\uplus t}$ is FID if and only if $0 \leq t \leq 1$. 

\item $\mathrm{\bf g}^{\uplus t}$ is classically infinitely divisible if and only if $t \in \{0,1 \}$. 
\end{enumerate}
\end{prop}
\begin{proof}

(1)\,\,\, Some properties shown below are known in \cite{BBLS11}, but we try to make this proof self-contained.  We are going to check Lemma \ref{prop323}(\ref{fdi1}). 
Let $f(y)$ denote the function $\frac{1}{i}F_{\mathbf{g}}(iy)$.  
The function $F_{\mathbf{g}}$ extends to $i\real$ analytically and does not have a pole  in $i\real$ since 
$F_{\mathbf{g}}(z) = \lim_{p\to\infty}F_{D_{\sqrt{2q}}(\mathbf{t}_q)}(z)$ locally uniformly in $i\real$. This convergence holds not only in $i(0,\infty)$ but also in $i(-\infty,0]$ 
by changing the contour $\real$ of the integral in $G_{D_{\sqrt{2q}}(\mathbf{t}_q)}$ to an arc in $\comp^-\cup \real$ as in Proposition \ref{prop78}. Because $ \frac{1}{i}F_{D_{\sqrt{2q}}(\mathbf{t}_q)}(iy) >0$ for $y \in (-\sqrt{2q},\infty)$ and it is strictly increasing as proved in Lemma \ref{lem129}, one also has $f(y)>0$ for $y\in\real$. 

The function $f$ satisfies the differential equation 
\begin{equation}\label{eq45}
f'(y) = f(y)^2 -yf(y)
\end{equation}
as proved in \cite[Eq.\ (3.6)]{BBLS11}, which also follows from a limit of Lemma \ref{lem11}(\ref{lem31}).  If $y >0$, then $f(y) >y$ from the basic property of a reciprocal Cauchy transform, and hence $f'(y) =f(y) (f(y) -y) >0$. If $y \leq0$, then $f'(y) >0$ from the fact $f(y)>0$ and (\ref{eq45}). Hence, $f'(y)>0$ for every $y\in\real$. 

We know that $f(\infty)=\infty$ from Proposition \ref{prop08}(\ref{p3}). Since $f$ is increasing, the limit $a:=\lim_{y \to -\infty} f(y)$ exists in $[0,\infty)$. If $a$ were strictly positive, then $f'(-\infty) =\infty$ from (\ref{eq45}). However $f(y) = f(0) -\int_{y}^0 f'(x)dx$, implying $f(-\infty) =-\infty$, a contradiction. Hence $a=0$.

(2) \,\,\, By shifting the contour by $-i$ , one can write 
$$
xG_{\mathbf{g}}(x) =\frac{1}{\sqrt{2\pi}} \int_{\real}\frac{x}{x-y+ i}e^{-\frac{1}{2}(y-i)^2}dy. 
$$ 
We divide the integral into two parts.   
First we find 
$$
\left|\frac{1}{\sqrt{2\pi}}\int_{\sqrt{x}} ^\infty \frac{x}{x-y+ i}e^{-\frac{1}{2}(y-i)^2}dy\right| \leq \frac{1}{\sqrt{2\pi}}\int_{\sqrt{x}} ^\infty y^2e^{-\frac{1}{2}y^2 +\frac{1}{2}}dy \to 0 \text{~as~} x \to \infty. 
$$ 
Next, we have $\sup_{y \in (-\infty, \sqrt{x}]}|\frac{x}{x-y+i} -1| \to 0$ as $x\to \infty$  and hence 
$$
\frac{1}{\sqrt{2\pi}}\int_{-\infty}^{\sqrt{x}} \frac{x}{x-y+ i}e^{-\frac{1}{2}(y-i)^2}dy \to 1 \text{~as~} x \to \infty
$$ 
from the dominated convergence theorem. By symmetry, we conclude $xG_{\mathbf{g}}(x) \to 1$ as $|x| \to \infty$. 

From the Stieltjes inversion formula, the density of $\mathbf{g}^{\uplus t}$ can be written as 
$$
\frac{t}{|(1-t)xG_{\mathbf{g}}(x) +t|^2}\frac{1}{\sqrt{2\pi}}e^{-\frac{x^2}{2}}. 
$$
For each $t>0$, the above density behaves like $\frac{t}{\sqrt{2\pi}}e^{-\frac{x^2}{2}}$ for large $|x|>0$ since $xG_{\mathbf{g}}(x) \to 1$. Any classically infinitely divisible distribution with Gaussian-like tail behavior must be exactly a Gaussian (see Corollary 9.9 of \cite{HS04}). Therefore, $\mathrm{\bf g}^{\uplus t}$ is not infinitely divisible for $ t \neq 0,1$. 
\end{proof}

Finally we show a general result which in particular enables us to compute the free divisibility indicator of the distribution $\mathbf{w}\boxplus \mathbf{g}$. This measure appeared as the spectral distribution of large random Markov matrices \cite{BDJ06}. 

\begin{prop}
%\begin{enumerate}[\rm(1)]
$\phi(\mathbf{w}\boxplus \mu)=1$ for any $\mu\in \mathcal{UI}_s$.  In particular,  $\phi(\mathbf{w} \boxplus \mathbf{g})=1$. 
%\item Let $\mathbf{m}(\lambda,\nu)$ denote the compound free Poisson distribution with rate $\lambda >0$ and jump distribution $\nu$. 
%Assume that $\nu$ is a symmetric probability measure satisfying $\int_{0}^1 \frac{1}{x^2}\,\nu(dx)=\infty$. Then $\phi(\mathbf{m}(\lambda,\nu) )=1$ for $0<\lambda\leq1$. 
%\end{enumerate}
\end{prop}
\begin{proof}
The measure $\mathbf{w} \boxplus\mu$ is clearly FID if $\mu \in \mathcal{UI}_s$. 
Since $\phi_{\mathbf{w}}(z)=\frac{1}{z}$, one gets the formula $F^{-1}_{\mathbf{w} \boxplus\mu}(z)=F^{-1}_\mu(z)+\frac{1}{z}$. 
Let us define $g(y):= \frac{1}{i}F^{-1}_{\mathbf{w} \boxplus\mu}(iy)= \frac{1}{i}F^{-1}_{\mu}(iy)-\frac{1}{y}$ for $y>0$. The assumption $\mu\in\mathcal{UI}_s$ implies that the map $y\mapsto \frac{1}{i}F^{-1}_{\mu}(iy)$ has positive derivative in $(0,\infty)$, so that $g'(y)>0$ for $y>0$ and $g(+0)=-\infty$. Moreover, $g(\infty)=\infty$ since $\lim_{y\to\infty}\frac{F^{-1}_\mu(iy)}{iy}=1$ (see \cite[Corollary 5.5]{BV93}). Hence there exists a real analytic compositional inverse of $g$ in $\real$, which extends $F_{\mathbf{w} \boxplus\mu}$ to a univalent map around $i\real$, mapping $i\real$ onto $i(0,\infty)$. Consequently, $\mathbf{w} \boxplus\mu$ satisfies the assumption of Lemma \ref{prop323}(\ref{fdi1}). It is proved in \cite{BBLS11} that $\mathbf{g}$ belongs to $\mathcal{UI}$ and hence to $\mathcal{UI}_s$, so the latter assertion holds. 
%(2)\,\,\, The free compound Poisson distribution $\mathbf{m}(\lambda,\nu)$ is characterized by the formula (see \cite{NS06, AHS})
%$$
%\phi_{\mathbf{m}(\lambda,\nu)}(z)= \lambda z (-1+zG_\nu(z)). 
%$$
%Let $h(y):=\frac{1}{i}F^{-1}_{\mathbf{m}(\lambda,\nu)}(iy)$ for $y>0$. Elementary calculus yields 
%$$
%h(y)= y-\lambda \int_{\real}\frac{x^2y}{x^2+y^2}\,\nu(dx),~~h'(y)= \int_{\real}\frac{(1-\lambda)x^4 +(2+\lambda)x^2y^2+y^4}{(x^2+y^2)^2}\,\nu(dx), 
%$$
%and so $h(+0)=0$, $h(\infty)=\infty$ and $h'(y)>0$ for $y>0$ provided $0<\lambda \leq 1$. We then have $h'(+0)=$This implies that $F_{\mathbf{m}(\lambda,\nu)}$ is univalent around $i(0,\infty)$, mapping $i(0,\infty)$ onto itself. Hence the assumption of Lemma \ref{prop323}(\ref{fdi2}) is satisfied for $c=0$. 
\end{proof}
\begin{rem}  
The semicircle law can be replaced by a symmetric free $\alpha$-stable law $\mathbf{f}_\alpha$ for any $\alpha \in(1,2]$ without difficulty. The measure $\mathbf{f}_\alpha$ is characterized by $\phi_{\mathbf{f}_\alpha}(z)=-e^{\frac{i \alpha\pi}{2}}z^{1-\alpha}$ for $\alpha\in(0,2]$ \cite{BV93, BP99} and so $\frac{1}{i}\phi_{\mathbf{f}_\alpha}(iy)=-y^{1-\alpha}$. In this context, it is known in \cite{AHa} that $\phi(\mathbf{f}_\alpha)=1$ for $\alpha \in(\frac{2}{3},2]$ and $\phi(\mathbf{f}_\alpha)=\infty$ for $\alpha \in(0,\frac{2}{3}]$. 
\end{rem}

\section*{Acknowledgement}
This work was supported by Marie Curie Actions--International Incoming Fellowships Project  328112 ICNCP and by Global COE program at Kyoto university. The author thanks Octavio Arizmendi and all referees for a lot of comments to improve this paper.

\end{document}